\documentclass[a4paper,11pt]{article}

\usepackage[T1]{fontenc} 
\usepackage[utf8]{inputenc} 
\usepackage[english]{babel}
\usepackage[left=1.5in,textwidth=30pc,textheight=54.5pc]{geometry} 

\usepackage{amsmath,amsfonts,amssymb,amsthm,mathrsfs}
\usepackage[hidelinks, breaklinks]{hyperref}
\usepackage{enumitem}
\usepackage{textcase} 
\usepackage{filemod} 
\usepackage[longnamesfirst]{natbib}
\usepackage{titlesec} 
\usepackage{etoolbox} 
\usepackage{mathtools}

\usepackage{subfig}
\usepackage{graphicx}

\makeatletter

\def\cite{\citet}

\numberwithin{equation}{section}
\allowdisplaybreaks

\def\@noindentfalse{\global\let\if@noindent\iffalse}
\def\@noindenttrue {\global\let\if@noindent\iftrue}
\def\@aftertheorem{%
  \@noindenttrue
  \everypar{%
    \if@noindent%
      \@noindentfalse\clubpenalty\@M\setbox\z@\lastbox%
    \else%
      \clubpenalty \@clubpenalty\everypar{}%
    \fi}}

\theoremstyle{plain}
\newtheorem{theorem}{Theorem}[section]
\AfterEndEnvironment{theorem}{\@aftertheorem}

\AfterEndEnvironment{definition}{\@aftertheorem}
\newtheorem{lemma}[theorem]{Lemma}
\AfterEndEnvironment{lemma}{\@aftertheorem}

\AfterEndEnvironment{corollary}{\@aftertheorem}
\newtheorem{proposition}[theorem]{Proposition}
\AfterEndEnvironment{proposition}{\@aftertheorem}

\theoremstyle{definition}
\newtheorem{remark}[theorem]{Remark}
\AfterEndEnvironment{remark}{\@aftertheorem}

\AfterEndEnvironment{example}{\@aftertheorem}
    
\AfterEndEnvironment{proof}{\@aftertheorem}

\titleformat{name=\section}
  {\center\small\sc}{\thesection\kern1em}{0pt}{\MakeTextUppercase}
\titleformat{name=\section, numberless}
  {\center\small\sc}{}{0pt}{\MakeTextUppercase}

\titleformat{name=\subsection}
  {\bf\mathversion{bold}}{\thesubsection\kern1em}{0pt}{}
\titleformat{name=\subsection, numberless}
  {\bf\mathversion{bold}}{}{0pt}{}

\titleformat{name=\subsubsection}
  {\normalsize\itshape}{\thesubsubsection\kern1em}{0pt}{}
\titleformat{name=\subsubsection, numberless}
  {\normalsize\itshape}{}{0pt}{}

\def\be#1{\begin{equation*}#1\end{equation*}}
\def\ben#1{\begin{equation}#1\end{equation}}
\def\bea#1{\begin{eqnarray*}#1\end{eqnarray*}}

\def\bes#1{\begin{equation*}\begin{split}#1\end{split}\end{equation*}}
\def\besn#1{\begin{equation}\begin{split}#1\end{split}\end{equation}}
\def\bs#1{\begin{split}#1\end{split}}
\def\ba#1{\begin{align*}#1\end{align*}}
\def\ban#1{\begin{align}#1\end{align}}
\def\bg#1{\begin{gather*}#1\end{gather*}}
\def\bgn#1{\begin{gather}#1\end{gather}}
\def\bm#1{\begin{multline*}#1\end{multline*}}
\def\bmn#1{\begin{multline}#1\end{multline}}

\def\note#1{\par\smallskip%
\noindent\kern-0.01\hsize%
\setlength\fboxrule{0pt}\fbox{\setlength\fboxrule{0.5pt}\fbox{%
\llap{$\boldsymbol\Longrightarrow$ }%
\vtop{\hsize=0.98\hsize\parindent=0cm\small\rm #1}%
\rlap{$\enskip\,\boldsymbol\Longleftarrow$}
}}%
}

\usepackage{delimset}

\def\given{\mskip 0.5mu plus 0.25mu\vert\mskip 0.5mu plus 0.15mu}
\newcounter{bracketlevel}%
\def\@bracketfactory#1#2#3#4#5#6{%
\expandafter\def\csname#1\endcsname##1{%
\global\advance\c@bracketlevel 1\relax%
\global\expandafter\let\csname @middummy\alph{bracketlevel}\endcsname\given%
\global\def\given{\mskip#5\csname#4\endcsname\vert\mskip#6}\csname#4l\endcsname#2##1\csname#4r\endcsname#3%
\global\expandafter\let\expandafter\given\csname @middummy\alph{bracketlevel}\endcsname%
\global\advance\c@bracketlevel -1\relax%
}%
}
\def\bracketfactory#1#2#3{%
\@bracketfactory{#1}{#2}{#3}{relax}{0.5mu plus 0.25mu}{0.5mu plus 0.15mu}
\@bracketfactory{b#1}{#2}{#3}{big}{1mu plus 0.25mu minus 0.25mu}{0.6mu plus 0.15mu minus 0.15mu}
\@bracketfactory{bb#1}{#2}{#3}{Big}{2.4mu plus 0.8mu minus 0.8mu}{1.8mu plus 0.6mu minus 0.6mu}
\@bracketfactory{bbb#1}{#2}{#3}{bigg}{3.2mu plus 1mu minus 1mu}{2.4mu plus 0.75mu minus 0.75mu}
\@bracketfactory{bbbb#1}{#2}{#3}{Bigg}{4mu plus 1mu minus 1mu}{3mu plus 0.75mu minus 0.75mu}
}
\bracketfactory{clc}{\lbrace}{\rbrace}
\bracketfactory{clr}{(}{)}
\bracketfactory{cls}{[}{]}
\bracketfactory{abs}{\lvert}{\rvert}
\bracketfactory{norm}{\Vert}{\Vert}
\bracketfactory{floor}{\lfloor}{\rfloor}
\bracketfactory{ceil}{\lceil}{\rceil}
\bracketfactory{angle}{\langle}{\rangle}

\newcounter{ctr}\loop\stepcounter{ctr}\edef\X{\@Alph\c@ctr}%
	\expandafter\edef\csname s\X\endcsname{\noexpand\mathscr{\X}}
	\expandafter\edef\csname c\X\endcsname{\noexpand\mathcal{\X}}
	\expandafter\edef\csname b\X\endcsname{\noexpand\boldsymbol{\X}}
	\expandafter\edef\csname I\X\endcsname{\noexpand\mathbb{\X}}
\ifnum\thectr<26\repeat

\let\@IE\IE\let\IE\undefined
\newcommand{\IE}{\mathop{{}\@IE}\mathopen{}}
\let\@IP\IP\let\IP\undefined
\newcommand{\IP}{\mathop{{}\@IP}}

\newcommand{\bigo}{\mathop{{}\mathrm{O}}\mathopen{}}
\newcommand{\lito}{\mathop{{}\mathrm{o}}\mathopen{}}
\newcommand{\law}{\mathop{{}\sL}\mathopen{}}

\newcommand{\pp}{\mathbb{P}}
\newcommand{\E}{\mathbb{E}}
\newcommand{\R}{\mathbb{R}}
\newcommand{\kU}{\mathcal{U}}
\newcommand{\kA}{\mathcal{A}}

\newcommand{\atanh}{\mathrm{atanh}}

\newcommand{\bp}{\mathbf{p}}

\newcommand{\bb}{\bar{\beta}}
\newcommand{\bh}{\bar{h}}
\newcommand{\hbe}{\hat{\beta}}
\newcommand{\tW}{\tilde{W}}
\newcommand{\tX}{\tilde{X}}

\newcommand{\bob}{\boldsymbol{\beta}}
\newcommand{\bop}{\boldsymbol{p}}

\let\original@left\left
\let\original@right\right
\renewcommand{\left}{\mathopen{}\mathclose\bgroup\original@left}
\renewcommand{\right}{\aftergroup\egroup\original@right}

\def\^#1{\relax\ifmmode {\mathaccent"705E #1} \else {\accent94 #1} \fi}
\def\~#1{\relax\ifmmode {\mathaccent"707E #1} \else {\accent"7E #1} \fi}
\def\*#1{\relax#1^\ast}
\edef\-#1{\relax\noexpand\ifmmode {\noexpand\bar{#1}} \noexpand\else \-#1\noexpand\fi}
\def\>#1{\vec{#1}}
\def\.#1{\dot{#1}}

\def\atop{\@@atop}
\def\%#1{\mathcal{#1}}

\def\mel{\MoveEqLeft}
\renewcommand{\leq}{\leqslant}
\renewcommand{\geq}{\geqslant}
\renewcommand{\phi}{\varphi}
\newcommand{\eps}{\varepsilon}
\newcommand{\D}{\Delta}
\newcommand{\N}{\mathop{{}\mathrm{N}}}

\newcommand{\I}{\mathop{{}\mathrm{I}}\mathopen{}}

\newcommand{\dw}{\mathop{d_{\mathrm{W}}}\mathopen{}}
\newcommand{\dk}{\mathop{d_{\mathrm{K}}}\mathopen{}}

\newcommand\indep{\protect\mathpalette{\protect\@indep}{\perp}}
\def\@indep#1#2{\mathrel{\rlap{$#1#2$}\mkern2mu{#1#2}}}

\newcommand{\longto}{\longrightarrow}
\def\tsfrac#1#2{{\textstyle\frac{#1}{#2}}}

\def\parsetime#1#2#3#4#5#6{#1#2:#3#4}
\def\parsedate#1:20#2#3#4#5#6#7#8+#9\empty{20#2#3-#4#5-#6#7 \parsetime #8}
\def\moddate{\expandafter\parsedate\pdffilemoddate{\jobname.tex}\empty}

\usepackage[linewidth=1pt]{mdframed}
\def\note#1{\begin{mdframed}#1\end{mdframed}}%

\usepackage[textwidth=1.64in]{todonotes}
\let\@@todo\todo
\def\todo#1{\@@todo[color=red,backgroundcolor=red!10,size=\tiny]{#1}}

\makeatother

\begin{document}
\title{\sc\bf\large\MakeUppercase{Mean-field spin models -- Fluctuation of the magnetization and maximum likelihood estimator}}
	\author{\sc Van Hao Can
		\footnote{Institute of Mathematics, Vietnam Academy of Science and Technology.
			\newline
			Email: cvhao@math.ac.vn. 
		}
		 \and \sc Adrian R\"ollin
	 \footnote{Department of Statistics and Data Science,  National University of Singapore.
	 	\newline
	 	Email: adrian.roellin@nus.edu.sg.
	 }
 }
	
	\date{}
	
	\maketitle
	
	\begin{abstract}
		Consider the mean-field spin models where the Gibbs measure of each configuration depends only on its magnetization. 	Based on the Stein and Laplace methods, we give a new and short proof for the  scaling limit theorems  with convergence rate for the magnetization in a perturbed  model. As an application, we derive the scaling limit theorems for the maximum likelihood estimators (MLEs) in  linear models. Remarkably, we characterize the full diagram of fluctuations for the magnetization and MLEs by analyzing the structure of the maximizers of a function associated with the Hamiltonian. For illustration, we apply our results to several well-known mixed spin models, as well as to the annealed Ising model on random regular graphs
  
	\end{abstract}

	\section{Introduction}
	The Ising model was originally proposed for the purpose to study the properties of ferromagnetic materials, but it has become since a prototype spin model on general graphs, see \cite{El,H,N1,N2}. Recently, it has also become a model for describing the pairwise interactions in networks, see e.g. \cite{CGa}, \cite{GG, GR} for its application in social networks, computer vision, and biology. However, in some situations, pairwise interaction is not enough to express the dependence of spins in networks, which motivated the study of higher-order Ising models, where multi-atom interactions are allowed; see for example \cite{Battetal, HBH,S,Yeal, OCPT}. Among various types of multi-spin interactions, three-body and four-body interactions have attracted particular interest from physicists due to their role in describing frustration in complex systems and their potential to enhance quantum computation. For further discussion, see \cite{KFRMC}

 Recently, mathematicians have begun rigorously studying certain mean-field models, including the cubic Ising model by \cite{CMO,E} and homogeneous $p$-spin model by \cite{MSB1,MSB2}. Both of these models can be formulated as a mixed spin model as follows.  Given the temperature parameters~$\bob=(\beta_1,\ldots,\beta_k)\in \IR^k$ and the order of mixed spin~$\bop=(p_1,\ldots,p_k) \in \IN^k$, the Gibbs measure is given by 
 \ben{ \label{mf}
		\mu_n(\omega) \propto \exp \left( H_{n, \bob, \bop} (\omega) \right), \quad \omega \in \Omega_n =\{1,-1\}^n,
	}
 where the  Hamiltonian $ H_{n, \bob, \bop} (\omega)$ is of mixed form:  
\ben{ \label{gibmf}
	 H_{n, \bob, \bop} (\omega)= \sum_{j=1}^k\frac{\beta_j}{n^{p_j-1}} \sum_{1 \leq i_1, \ldots, i_{p_j} \leq n} \omega_{i_1} \ldots \omega_{i_{p_j}}  = n f_{\bop,\bob}(\-\omega), 
}
where
	\be{ 
\bar{\omega} = \frac{\omega_1 + \ldots+ \omega_n}{n}, \qquad		f_{\bop,\bob}(t)=\sum_{i=1}^k\beta_i t^{p_i}.
	}
Note that here and below, for any measure~$\mu$, the notation~$\mu(\omega) \propto f(w)$ means that the value of~$\mu(\omega)$ is proportional to~$f(\omega)$ up to a normalising constant that only depends on the model parameters. In \eqref{mf} and \eqref{gibmf}, all possible~$p_j$-tuples with $j=1,\ldots,k$ in the complete graph of size~$n$ contribute to the Hamiltonian, with each tuple interacting with a different strength parameter.  In the cubic model with $\bop=(p_2,p_3)$, \cite{CMO} offer the complete phase diagram of parameter $\bob=(\beta_2,\beta_3)$, which  determines the scaling limit of the magnetization. Later, \cite{E} establishes  the rate of convergence of limit theorems.  On the other hand, the  $p$-spin model  corresponds to the case  $\bop=(p,1)$ i.e. only the $p$-spin interaction and external field are considered.   \cite{MSB1} and \cite{MSB2} investigate the fluctuation of the magnetization~$\bar{\omega}$, as well as the maximum likelihood estimators for the parameters~$\beta_p$ and~$\beta_1$.  
	
	In this article, we go further and study the fluctuation of the magnetisation in the case where the interaction can be expressed by a general (smooth enough) function of~$\bar{\omega}$ instead of just being a polynomial of~$\bar{\omega}$ as in \eqref{gibmf}. Consider the generalized linear model
	\ben{ \label{3}
		\mu_n(\omega) \propto \exp \bbclr{n\bclr{\beta_1f_1(\-\omega_+)+\ldots+\beta_l f_l(\-\omega_+)} }, \quad \-\omega_+ =\frac{\babs{ \{i: \omega_i=1\}}}{n}, 
	}
	where~$f_1,\dots,f_l$ are smooth functions and~$\beta_1,\dots,\beta_l$ are real-valued model parameters. Note that~$\-\omega_+= (\bar{\omega}+1)/2$, and so studying~$\bar{\omega}$ and~$\-\omega_+$ is equivalent. 	
	Denoting $X_n=\abs{\{i:\omega_i=1\}}=n\-\omega_+$, the linear model \eqref{3} can be characterized by the simpler model
	\be{
\IP[X_n=k] \propto \exp \big(n F(k/n)\big) \binom{n}{k}, \quad 
 0 \leq k \leq n,	
}
where 
\be{
	F(a) = \beta_1 f_1(a) + \ldots+ \beta_l f_l(a), \qquad a \in [0,1].
}
Observe further that 
\be{
\frac{1}{n} \log \binom{n}{k} \approx I(k/n), \quad I(a)= -a \log a + (a-1) \log (1-a),
}	
	here, $I$ is the entropy function. 	Combining Stein's method for normal approximation and Laplace's method, we derive a complete description of the fluctuation of~$\-\omega_+$ (and thus, $\bar{\omega}$). It turns out that the order of the fluctuation depends on the order of regularity of the maximizers of an associated  function~$A: [0,1] \rightarrow \R$ given as
	\ben{ \label{doa}
		A(a) = F(a) + I(a).
	}
More precisely, suppose that $A$ has finite maximizers $(a_j)_{j\in J}$ and that each  $a_j$ is $2m_j$-regular for $j \in J$. Here, a maximizer $a_* \in (0,1)$ is called ~$2m$-regular  if~$A^{(k)}(a_*)=0$ for~$1\leq k \leq 2m-1$ and if~$A^{(2m)}(a_*)<0$. Then our general result implies that~$\-\omega_+$ concentrates around maximizers with highest  regularity order~$(a_j)_{j \in J_1}$, where $J_1=\{j \in J: m_j = \max_{i \in J} m_i \}$. Moreover, for all $j \in J$,   conditionally on $\-\omega_+ \in (a_j-\delta, a_j+\delta)$, the scaled magnetization $(\-\omega_+-a_j)n^{1/(2m_j)}$ converges in law to a random variable $Y_j$  whose the density is proportional to $\exp(-c_jx^{2m_j})$ where $c_j$ depends on $A^{(2m_j)}(a_j)$. Consequently, the complexity of the maximizers of $A$ results in a  diverse phase diagram for the magnetization. 
	
	The second question we address in this article is the construction of suitable estimators of the model parameters. The maximum likelihood estimators (MLEs) in the~$p$-spin Curie-Weiss model was studied by \cite{CG} for~$p=2$ and by \cite{MSB2} for~$p\geq 3$, and for Markov random fields on lattices by \cite{Co,P}. The maximum pseudo likelihood estimation problem of the Ising model on general graphs has been discussed by \cite{Ch} and \cite{GM}. We refer to \cite{MSB2}  and the references therein for further discussion on the history and development of the problem.
	
	In this article, we follow the usual approach to construct the MLE for each parameter~$\beta_i$ using only one sample~$\omega$. In fact, we can construct a consistent estimator~$\hbe_{i,n}$ of~$\beta_i$ using only the quantity~$\-\omega_+$; see more in Section~\ref{sec2}. Apart from consistency, we can also show that, after suitable scaling, $\hbe_{i,n}-\beta_i$ converges to a non-degenerate random variable. A standard approach to study the fluctuation and scaling limits of~$\hbe_{i,n}$ is to prove limit theorems for a perturbed model of \eqref{3}; see for example \cite{CG} and \cite{MSB2} for~$p$-spin Curie-Weiss models. 
	
	In the general setting,   we consider the perturbed model 
	\ben{
		\IP[X_n=k] \propto \exp \bbclr{nA_n(k/n)+n^{1/(2m)} B_{n}(k/n)}, \quad  0\leq k \leq n, \label{4}
	}
	where~$A_n, B_n: \{0,\tfrac{1}{n}, \ldots, 1\} \rightarrow \R$; here, $A_n$ is the main term driving the model and~$B_n$ is the perturbation. We assume in addition  that~$A_n$ and~$B_n$ are well approximated by smooth functions~$A, B: [0,1] \rightarrow \R$ and~$2m$ is the regularity order of the maximizers of $A$. Particularly,  for the linear model \eqref{3}, the knowledge of the fluctuation of~$X_n$  with~$A$ given by \eqref{doa} and~$B$ suitably chosen would lead to the scaling limit of estimators~$\hbe_{1,n}, \dots,\hbe_{l,n}$ of the linear model \eqref{3}. We refer to Section~\ref{sec2} for detailed proofs.

	The usual strategy to investigate the Gibbs measure of the form \eqref{3} (or the more general form \eqref{4}) is using Laplace's method to prove the concentration and scaling limit of magnetization around maximizers of~$A(a)$. This approach usually requires many tedious and difficult computations of exponential functionals. Our main innovation in the study of the perturbed model \eqref{4} is  exploiting  Stein's method to avoid some of these complicated computations. Moreover, as an additional bonus of using Stein's method, we also obtain the rate of convergence in our limit theorems. We refer to Section~\ref{sec1} for more details. 
	
	We briefly summarize the main findings of this paper.
	\begin{itemize}
          \item[$\triangleright$] Main Theorems (Theorems \ref{thm1}–\ref{thm3}): We provide general sufficient conditions \ref{A1}–\ref{A4} for the Hamiltonian of mean-field models under which the law of large numbers, concentration inequalities, and conditional scaling limit theorems hold. Among these, condition \ref{A1} plays a crucial role in determining the phase diagram for the limit theorems. As a result, studying mean-field spin models can now be reduced to describing condition \ref{A1}, or understanding the maximizers of the associated function  $A$.

          More specifically, in Theorems~\ref{thm1}--\ref{thm2}, using Laplace's methods, we obtain the strong law of large numbers and concentration inequalities for the magnetization (or for~$X_n$).  In Theorem~\ref{thm3}, using Stein's method we  give a concise proof for distributional  limit theorem of the magnetization  with convergence rate in Wasserstein distance. 
		\item[$\triangleright$] Application to MLEs (Theorem~\ref{thm4}): By applying the limit theorems for magnetization in perturbed models, we establish the scaling limits of maximum likelihood estimators for the linear model in \eqref{3}. 
     \item[$\triangleright$] Application to specific models: A significant part of Section 5 is devoted to investigating particular mixed spin models using our main theorems (\ref{thm1}–\ref{thm4}). We demonstrate a rich phase diagram for the scaling limits of magnetization, inherited from the complex structure of the maximizers of the associated function $A$  as presented in the condition~\ref{A1}. In addition to mixed spin models, we also apply our results to the annealed Ising model on random regular graphs.

   \end{itemize}

\subsection{Notation}

For any random variables~$X$ and~$Y$, we consider the Kolmogorov and Wasserstein probability metrics, defined as
\ba{
		\dk(X,Y) & = \sup_{t \in \R} |\IP[X \leq t]-\IP[Y \leq t]|,\\
		\dw (X,Y) & = \sup_{\norm{h'}\leq 1} \abs{\IE h(X) -\IE h(Y)}.
} 
For~$a>0$, we denote by $N^+(0,a)$ (resp.~$N^{-}(0,a)$) the positive (resp. negative) half-normal distribution, that is the distribution of $|Z|$ (resp.~$-|Z|$), where $Z \sim N(0,a)$.  	Let~$X$ be a random variable with density~$p(x)$. We write~$p(x)\propto f(x)$ if~$p(x)$ is proportional to~$f(x)$ up to a normalizing constant, and in such a case, we also write~$X \propto f(x)$ if~$X$ has distribution with density given by~$p(x)$. Let~$f$ and~$g$ be two real functions. We write~$f=\bigo(g)$ if there exists a universal constant~$C>0$ such that~$f(x) \leq C g(x)$ for all~$x$ in the domain of~$f$ and~$g$. We also write~$f=g+\bigo(h)$ when~$|f-g|=\bigo(|h|)$, and write~$f=\exp(g+\bigo(h))$ if~$|\log f - g| =\bigo(|h|)$.  In some cases, we write~$f=\bigo_{\delta}(g)$ to emphasize that the constant~$C$ may depend on~$\delta$.
  
\section{The magnetization in perturbed models} \label{sec1}

Let~$A_n, B_n: \{0,1/n, \ldots, 1\} \rightarrow \R$ and~$m_* \in \IN$. We consider the integer-valued random variable~$X_n$ defined by the model
\be{
	\IP[X_n=k] = \frac{1}{Z_n} \exp \bclr{H_n(k/n)}, \quad  0 \leq k  \leq n, 
}
where 
\bg{
	\quad H_n(k/n) =nA_n(k/n)+n\sigma_{*,n} B_{n}(k/n), \quad  \sigma_{*,n} =n^{-1+1/(2m_*)}, \\
	Z_n= \sum_{k=0}^{n}\exp \bclr{H_n(k/n)}.
}
In what follows, we will make use of various technical assumptions. Let~$\eps_*$, $\delta_*$, and~$C_*$ be positive constants, let~$(a_j,m_j)_{j\in J}$ be a finite collection of pairs with~$a_j\in(0,1)$ and~$m_j\in \IN$ for~$j\in J$, and let~$A,B:[0,1] \rightarrow \IR$ be functions such that~$A \in C^{2m_*+1}([0,1])$ and~$B \in C^2([0,1])$. Consider the following assumptions:
\begin{enumerate}[label=\textrm{(A\arabic*)}]
	\item\label{A1}~$(a_j)_{j \in J}$ are all the maximizers of~$A$, and~$\max_{j\in J}m_j=m_*$.  We have~$A'(a_j)=\ldots=A^{(2m_j-1)}(a_j)=0$ and~$\max_{|x-a_j| \leq \delta_*} A^{(2m_j)}(x)<0$ for all~$j\in J$.  The intervals~$(a_j-\delta_*,a_j+\delta_*)$, $j\in J$, are disjoint and contained in~$(0,1)$. 

  \item\label{A2} For~$n$ large enough and for all~$k$ for which~$|k/n-a_j| \geq \delta_*$ for all~$j \in J$, we have
		\[ A_n(k/n) \leq \max_{x \in [0,1]} A(x) -\varepsilon_*, \quad |B_n(k/n)| \leq C_*. \]
  \item\label{A3}  For~$n$ large enough and for all~$k, \ell$  for which there is~$j\in J$ such that~$|k/n-a_j|, |\ell/n-a_j| < \delta_*$, we have
		\ben{ \label{A3i}
			|A_n(k/n) -  A(k/n)| \leq \frac{C_* \log n}{n}; \quad   |B_{n}(k/n) -  B(k/n)| \leq  \frac{C_*}{n} \tag{A3i}
		}
		and 
		\ben{ \label{A3ii}
			\babs{[A_n \left(k/n\right) - A_n \left(\ell/n\right)] - [A \left(k/n\right) - A \left(\ell/n\right)]} \leq  \frac{C_*|k-\ell|}{n^2}. \tag{A3ii}  
		}
  \item\label{A4} Let 
  	\be{
	J_1= \{j \in J: B(a_j) = \max_{k \in J} B(a_k) \}, \quad 	J_2=	 \{ j \in J_1: m_j = \max_{k \in J_1} m_k  \}. 
	  }	
   Then there exist real numbers $(\nu_j)_{j \in J_2}$ such that for~$n$ large enough and for~$k_j=[na_j]$, $j \in J$,  we have
		\be{
			\sup_{i, j \in  J_2} \babs{n[A_n( k_i/n)-A_n( k_j/n)]-(\nu_i-\nu_j)} \leq  \frac{C_*}{n\sigma_{*,n}}.
		} 
\end{enumerate}

\begin{remark}  \label{rem:a4}
  Let us  examine the  \ref{A1}--\ref{A4} for a typical class of  mean field model, where  
\[
\mu_n(\omega) \propto \exp \left(nF(\-\omega_+) + n \sigma_{*,n} B(\-\omega_+) \right),
\]
where $F$ and  $B$ are functions in $C^{2m_*+1}([0,1])$ and $C^2([0,1])$ respectively. We then have 
\be{
A_n(k/n)= F(k/n) + \frac{1}{n} \log  \binom{n}{k}.
}
By the Stirling formula, $n!=\sqrt{2 \pi n} (\tfrac{n}{e})^n(1+\bigo(n^{-1}))$. Hence, given $\delta \in (0,1)$,  for all $n \delta \leq k \leq n(1-\delta)$   
\ben{ \label{stir}
\frac{1}{n} \log  \binom{n}{k} = \frac{1}{n}  \log \sqrt{\frac{n}{2\pi(n-k)k}} + I(k/n) +  \bigo_\delta(n^{-2}).
}
Let 
\[
A(t) = F(t) + I(t).
\]
Since $I'(0)=\infty$ and $I'(1)=-\infty$, there exists a positive constant $\delta_*$ such that  the set of maximizers of $A$, denoted by $(a_j)_{j \in J}$, lies in $[\delta_*,1-\delta_*]$ and satisfies the condition (A1). Using the approximation \eqref{stir}, the conditions \ref{A2} and \eqref{A3i} can be easily verified. For \eqref{A3ii}, using \eqref{stir} and the inequality that $|log x - \log y| \leq |x-y|/ \min\{x,y\}$, we have 
\ba{
	\babs{[A_n \left(k/n\right) - A_n \left(\ell/n\right)] - [A \left(k/n\right) - A \left(\ell/n\right)]} &=  \bigo_{\delta_*}(n^{-2}) + \frac{1}{n} \babs{\log \frac{k(n-k)}{\ell(n-\ell)}} \\
	& =\bigo_{\delta_*}(1) \frac{|k-\ell|}{n^2}.
}
Finally,  define for $j \in J$
\ben{\label{nuj}
\nu_j=\log \sqrt{\frac{1}{(1-a_j)a_j}}. 
}
Then using \eqref{stir} again, we have 
\be{
			\sup_{i, j \in  J} \babs{n[A_n( k_i/n)-A_n( k_j/n)]-(\nu_i-\nu_j)} =\bigo(n^{-1}), 
		}
and thus \ref{A4} is satisfied.
\end{remark}

\begin{theorem}[Weak law of large numbers] \label{thm1} Under Assumptions \textnormal{\ref{A1}--\ref{A4}},
we have
\ben{ \label{5}
\frac{X_n}{n} \stackrel{\law}{\longto} \sum_{j \in J_2} p_j \delta_{a_j}, 
}
where for~$j \in J_2$,
\be{
p_j = \frac{q_j e^{\nu_j}}{\sum_{k \in J_2} q_k e^{\nu_k}}, \qquad q_j = \int_{\IR} \exp(c_jx^{2m_j} + b_j x)dx,  
}
with 
\ben{ \label{6}
c_j =\frac{A^{(2m_j)}(a_j)}{(2m_j)!}, \qquad b_j = B'(a_j)\I[m_j=m_*]. 
}
\end{theorem}

\begin{theorem}[Concentration] \label{thm2} Assume \textnormal{\ref{A1}--\ref{A3}}, and let~$\delta \in (0,\delta_*)~$. There exist a positive constants~$c$ such that
 \ben{ 
 	\pp  \left[\text{$|X_n/n-a_j| > \delta$ for all~$j \in J$} \right] \leq \exp(-cn)\label{7}
  }
and 
 \ben{ 
 	\pp \left[\text{$|X_n/n-a_j| > \delta$ for all~$j \in J_1$} \right] \leq \exp(-cn\sigma_{*,n}).\label{8}
}
Moreover, for any~$j_2\in J_2$, there exists a constant~$C$ such that if~$J_1 \neq J_2$,
\ben{
  \pp \left[\text{$|X_n/n-a_j| > \delta$ for all~$j \in J_2$} \right] \leq C \max_{j_1 \in J_1 \setminus J_2} n^{1/(2m_{j_1})-1/(2m_{j_2})},\label{9}
}
and, for any~$j \in J_2$, 
\ben{ \label{10}
\pp[|X_n/n-a_j| \leq \delta_*] = p_j + \bigo(\tau_{*,n}) + \bigo\bbbclr{\max_{k \in J_1 \setminus J_2} n^{1/(2m_k)-1/(2m_{j})}},
}
where
\be{
\tau_{*,n} =  \frac{(\log n)^{2m_*+1}}{n^{1/(2m_*)}}  +  n^{1/(2m_*)-1/(2m_{j_2})} \log n \I[m_{j_2} \neq  m_*].
}
\end{theorem}

\begin{theorem}[Distributional limit theorem] \label{thm3}
Under Assumptions \textnormal{\ref{A1}--\ref{A3}}, we have for all~$j \in J$ and~$l \in \IN$ that
\be{
	\IE\bclc{|X_n/n-a_j|^l \given |X_n-na_j| \leq n\delta_*} =\bigo\bclr{n^{-l/(2m_j)}};
}
and for all~$j \in J$ that
\bm{
	\dw \bclr{\law\bclr{n^{1/(2m_j)}(X_n/n-a_j)\given |X_n-na_j| \leq n\delta_*}, \law(Y_j)} \\
	=\bigo\bclr{n^{-1/(2m_j)}} + \bigo\bclr{n^{1/(2m_*)-1/(2m_j)}  \I[m_j \neq  m_*]},
}
where~$Y_j \propto \exp(c_jx^{2m_j}+b_j x)$ with~$c_j$ and~$b_j$ given as in \eqref{6}.
\end{theorem}

\begin{remark} In Theorem~\ref{thm1}, Assumption~\ref{A4} is not necessary when $|J_2|=1$ (particularly when~$A$ has a unique maximizer). In fact,~\ref{A4} is only required in \eqref{17} to prove \eqref{5}, where we  compare the Gibbs measure around the maximizers.   
\end{remark}

\section{Proofs of main results} 

To simplify notation,  we will drop the dependence on~$n$ in what follows and write~$X$, $W$, $\sigma$ and~$\tau$ instead of~$X_n$, $W_n$, $\sigma_n$ and~$\tau_n$, and introduce some notation 
\be{
	\sigma_j = n^{1/(2m_j)-1}; \qquad J_*=\{j \in J: \sigma_j =\sigma_* \} =\{j \in J: m_j =m_* \}.	
}
In order to prove Theorems~\ref{thm1},~\ref{thm2} and~\ref{thm3}, the following result is key. 
\begin{proposition} \label{prop1}
	Assume \textnormal{\ref{A1}--\ref{A3}}, and let~$\delta \in (0,\delta_*]$. Then for all~$j \in J$, we have
	 \bes{ 
		Z_{n,j}(\delta)&:= \sum_{|k/n-a_j| \leq \delta} \exp(H_n(k/n)) \\
		&= (q_j +\bigo_{\delta}(\tau_j)) \sigma_j^{-1} \exp(nA_n(k_j/n)+n\sigma_* B(a_j)),
	}
	where
	$k_j=[na_j]$ and~$\tau_j$, $q_j$, $c_j$ and~$b_j$ are given in Theorem~\ref{thm1}.	
\end{proposition}
The proof of Proposition~\ref{prop1} is based on Laplace's method and will be presented at the end of this section.

\subsection{Concentration and weak law of large numbers}

\begin{proof}[Proof of  Theorems~\ref{thm1} and~\ref{thm2}]
	 We start by proving the concentration inequalities. We first show  that  for any~$\delta \in (0,\delta_*)$, one has 
\ben{ \label{11}
	\pp \bcls{\text{$|X/n-a_j| > \delta$ for all~$j \in J$} } \leq \exp(-cn),
}
where~$c=c(\delta)>0$ is  a constant.  Let~$k$ be an integer such that~$|k-na_j| \geq \delta n$ for all~$j \in J$.  We claim that there exist~$i \in J$ and~$c>0$, such that
\ben{ \label{12}
A_n(k/n) -A_n(k_i/n) \leq -c,
}
where recall that~$k_i=[na_i]$.  Indeed, if~$|k-na_j| \geq \delta_* n$ for all~$j \in J$ then let~$i$ be an arbitrary element of~$J$ and using by~\ref{A2} and~\eqref{A3i}, we have
\bes{
	\mel A_n(k/n)-A_n(k_i/n)\\
	&=A_n(k/n)-A(a_i) + A(a_i)- A(k_i/n) + A(k_i/n)-A_n(k_i/n)\\
	&\leq -\eps_*+ \bigo(|k_i/n-a_i|) +\bigo( \log n/n) \leq -2\eps_*/3,
}
where we have used~$|k_i/n-a_i| \leq 1/n$. Otherwise, suppose that~$|k-na_i| \leq \delta_* n$ for some~$i \in J$. Then 
\bes{
  \mel A_n(k/n)-A_n(k_i/n) \\
  & = A(k/n)-A(k_i/n) + \bigo(1/n)
  = A(k/n)-A(a_i) + \bigo(1/n) \\
&\leq \sup_{x: |x-a_i| \leq \delta_*} A^{(2m_i)}(x) \delta_*^{2m_i}/(2m_i!) + \bigo(1/n)
\leq -c,
}
where~$c=c(\delta_*) >0$. Here, for the first two equations, we used~\eqref{A3i} and~$|k_i/n-a_i| \leq 1/n$, for the remaining inequalities, we used Taylor expansion and~\ref{A1}. The proof of \eqref{12} is complete.

Next, note that by~\ref{A2}, $|B_n(k/n)| \leq C_*$ when~$|k/n-a_j| \geq \delta_*$ for all~$j \in J$,   and by~\ref{A3} for~$k$ such that~$|k/n-a_j| \leq \delta_*$ for some~$j \in J$ one has~$|B_n(k/n)| \leq |B(k/n)|+C_*/n \leq 2 \max_{x \in [0,1]} |B(x)|$. Therefore,
\ben{ \label{13}
\max_{0 \leq k \leq n} |B_n(k/n)| = \bigo(1).
} 
 Combining \eqref{12} and \eqref{13} yields that for all~$n$ sufficiently large 
\besn{ \label{14}
	\mel H_n(k/n)-H_n(k_i/n) \\
	&=n[A_n(k/n) -A_n(k_i/n)]+n \sigma_* [B_n(k/n)-B_n(k_i/n)]\\
	&\leq -cn +\bigo(n\sigma_*)  \leq -cn/2,
}
and thus
\be{
	\IP[X=k] \leq  \exp(-cn/2) \IP[X=k_i] \leq \exp(-c n/4),
}
and \eqref{11} is proved by using the union bound. 

By Proposition~\ref{prop1},  for any fixed~$\delta \in (0,\delta_*)$, for all~$j \in J$ and~$n$ sufficiently large
\besn{ \label{15}
Z_{n,j}(\delta)&:= \sum_{|k/n-a_j| \leq \delta} \exp(H_n(k/n)) \\
&= (q_j +\bigo_{\delta}(\tau_j)) \sigma_j^{-1} \exp(nA_n(k_j/n)+n\sigma_* B(a_j)),
}
where
~$k_j=[na_j]$ and 
\be{ 
  \tau_j = \frac{(\log n)^{2m_*+1}}{n\sigma_*}  + \frac{\sigma_* \log n}{\sigma_j} \I[j \in J\setminus J_*],  
}
and
\be{
q_j =  \int_{\IR} \exp(c_jx^{2m_j} + b_j x)dx,
}
with~$c_j, b_j$ as in \eqref{6}. 

Note that~$nA_n(k_j/n) =n A(a_j) +\bigo( \log n) = n \max_{x \in [0,1]} A(x) +\bigo( \log n)$ by~\eqref{A3i}. Therefore, the leading terms of~$(Z_{n,j})_{j \in J}$ are the ones at which the sequence~$(B(a_j))_{j \in J}$ attains the maximum. Recall that 
\be{
J_1=\{j\in J: B(a_j) =\max_{k\in J} B(a_k) \}.
}
 Let~$\delta \in (0,\delta_*)$ be any fixed constant. By the above,   \eqref{11} and  \eqref{15} yield that,  if~$J_1 \neq J$,
\bes{
	\mel\pp\bcls{ \text{$|X/n-a_j| > \delta$  for all~$j \in J_1$}} \\
	& \leq \exp(-cn) + \frac{\sum_{j  \in J\setminus J_1} Z_{n,j}}{\sum_{j \in J} Z_{n,j}}  \\
	& \leq \exp(-cn) + \bigo_{\delta}(1) \sum_{j \in J \setminus J_1} \frac{\sigma_{j_1}}{\sigma_j} \exp(n\sigma_*(B(a_j)-B(a_{j_1}))) \\
	& \leq \exp(-c_1n\sigma_*),
}
where~$c$ and~$c_1$ are positive constants depending on~$\delta$, and~$j_1$ is an element of~$J_1$. Similarly, if~$J_2 \neq J_1$,
\ben{ \label{16}
\pp \bcls{\text{$|X/n-a_j| > \delta$ for all~$j \in J_2$}} \leq \bigo_{\delta}(1) \max_{j_1 \in J_1 \setminus J_2} \sigma_{j_2}/\sigma_{j_1},
}
with~$j_2$ an element of~$J_2$.  The two above inequalities and \eqref{11} yields the concentration estimates in \eqref{7}, \eqref{8} and \eqref{9}.

 We now prove  the weak law of large numbers  \eqref{5} and the estimate \eqref{10}. By~\ref{A4} for all~$i,j \in J_2$
\ben{ \label{17}
|nA_n(k_i/n)-nA_n(k_j/n)-(\nu_i-\nu_j)| = \bigo(1/n\sigma_*). 
}
Hence, it follows from \eqref{15} that  for any~$\delta \in (0, \delta_*)$, and for all~$i, j \in J_2$
\bea{
\frac{Z_{n,i}(\delta)}{Z_{n,j}(\delta)} &=& \left(1+ \bigo_{\delta}(\tau_*)  \right) \frac{q_i}{q_j} \exp(n[A_n(k_i/n)-A_n(k_j/n)])\\
&=&  \left(1+ \bigo_{\delta}(\tau_*)  \right) \frac{q_i}{q_j}   \exp(\nu_i-\nu_j+ \bigo(1/n\sigma_*))\\
&=& \left((1+ \bigo_{\delta}(\tau_*)    \right) \frac{q_i e^{\nu_i}}{q_je^{\nu_j}}, 
}
where 
\be{
\tau_*=  \tau_{j} = \frac{(\log n)^{2m_*+1}}{n\sigma_*}  +  \frac{\sigma_* \log n}{\sigma_{j}} \I[J_2 \neq J_*],
}
(note here that~$\sigma_{i}=\sigma_{j}$ and hence $\tau_i=\tau_j$). 
Therefore,
\ben{ \label{18}
\frac{Z_{n,j}(\delta)}{\sum_{k \in J_2} Z_{n,k}(\delta)} = p_j + \bigo_{\delta}(\tau_*),
}
where 
\be{
	p_j =\frac{q_je^{\nu_j}}{\sum_{k \in J_2} q_k e^{\nu_k} }.
}
Combining \eqref{18} and \eqref{16}, we have 
\be{
X/n \stackrel{\law}{\longto} \sum_{j \in J_2} p_j \, \delta_{a_j},
}
and for all~$j \in J_2$
\be{
\pp[|X/n-a_j| \leq  \delta_*] = p_j +\bigo(\tau_*) + \bigo(1)\max_{j_1 \in J_1 \setminus J_2} \sigma_{j_2}/\sigma_{j_1}.
}
The proof of \eqref{5} and \eqref{10} is complete. 
\end{proof}

\subsection{Stein's method}

We first state and derive what is needed to implement Stein's method for target distributions of the form~$p(y)\propto\exp(cy^{2m}+by)$. The following result is a consequence of the general approach of \cite{CS}.

	\begin{lemma}  \label{lem1}
		Let~$m$ be a positive integer, and let~$Y$ be a random variable with density function~$p(y)\propto\exp(cy^{2m}+by)$ with~$c<0$ and~$b \in \IR$. Then there exists a positive constant~$K=K(c,b,m)$ such that  for any random variable~$W$,
		\be{
			\dw(W,Y) \leq  \sup_{f \in C^{2}_K(\IR)} \Big|\IE\Big\{f'(W)+\tfrac{p'(W)}{p(W)}f(W) \Big \} \Big |,	
		}
		where 
		\be{
			C^2_K(\IR) =\bclc{f \in C^2(\IR): \norm{f}_{\infty}, \norm{f'}_{\infty},\norm{f''}_{\infty} \leq K }, 
		}
		with~$C^2(\IR)$ the space of twice differentiable functions and~$\norm{g}_{\infty} =\sup_{x \in \IR} |g(x)|$.	
	\end{lemma}
	
	\begin{proof} Let~$h$ be Lipschitz continuous and consider the Stein equation
		\ben{ \label{19}
			f'(w)+p'(w)f(w)/p(w) =h(w) -\IE h(Y).
		}
		\cite[Lemma 4.1]{CS} showed that the solution~$f_h$ of the functional equation \eqref{19} belongs to~$C^2(\IR)$ and satisfies 
		\be{
			\norm{f}_{\infty}\vee \norm{f'}_{\infty}\vee \norm{f''}_{\infty} \leq (1+d_1)(1+d_2)(1+d_3)\norm{h'}_{\infty},
		}
		where 
		\be{
			d_1 = \sup_{x \in \IR} \frac{\min \{P(x),1-P(x)\}}{p(x)}, 
			\quad d_2 = \sup_{x \in \IR} \frac{\min \{P(x),1-P(x)\}p'(x)}{p^2(x)},
		}
		and~$d_3 = \sup_{x \in \IR}  Q(x)$,  with~$P(x)=\int_{-\infty}^x p(t)dt$ and 
		\bm{
			Q(x) 
			= \frac{1+|(p'/p)'(x)|}{p(x)}\min\bclc{\IE\clc{Y\I[Y \leq x]} + \IE|Y| P(x),\\
				\IE\{Y\I[Y > x]\} + \IE|Y| (1-P(x))}.
		}
		We now show that~$d_3$ is a finite constant depending only on~$c$, $b$ and~$m$. The proof for~$d_1$ and~$d_2$ is  similar but simpler, hence omitted. It is clear that 
		\ben{ \label{20}
			d_3 = \max \bbbclc{\sup_{x \leq -C} Q(x), \sup_{|x| \leq C} Q(x), \sup_{x \geq C} Q(x) }, 
			\quad C= 1+ \frac{4 + |b|}{m|c|}.
		}
		First, consider~$x\geq C$; since~$(p'/p)'(x)=2m(2m-1)cx^{2m-2}$ and~$\IE Y <\infty$,
		\ben{ \label{21}
			Q(x) \leq C_1 \frac{x^{2m-2} \int_x^{\infty} y p(y) dy}{p(x)} = C_1 \frac{x^{2m-2} \int_x^{\infty} y q(y) dy}{q(x)},
		}
		with~$C_1=C_1(c,b,m)$ a finite constant and~$q(x) = \exp(cx^{2m}+bx)$.
		Using integration by parts and the fact that~$q'(y)=q(y)(2mc y^{2m-1} +b)<0$ for~$y \geq x\geq C$,
		\bes{
			\int_x^{\infty} y q(y) dy &= \int_x^{\infty} \frac{y}{2mc y^{2m-1} +b} d(q(y)) \leq \int_x^{\infty} \frac{y^{2-2m}}{mc} d(q(y))\\
			&=\frac{x^{2-2m}q(x)}{m|c|} + \int_x^{\infty} \frac{y^{1-2m}(2-2m)}{mc} q(y) dy\\
			&\leq \frac{x^{2-2m}q(x)}{m|c|} + \frac{1}{2} \int_x^{\infty}  yq(y) dy,
		}
		and hence
		\be{
			\int_x^{\infty} y q(y) dy \leq \frac{2x^{2-2m}q(x)}{m|c|}.
		}
		Combining this with \eqref{21} we have~$\sup_{x \geq C} Q(x) \leq 2C_1/(m|c|)$. The same inequality holds for~$\sup_{x \leq -C} Q(x)$. Since~$Q$ is continuous, it also follows that~$\sup_{|x| \leq C} Q(x) < \infty$ . Hence, by \eqref{20}, we have~$d_3<\infty$.
		
		Finally, considering \eqref{19} with~$w$ replaced by~$W$ and taking expectation, the claim easily follows.
	\end{proof}

	\begin{lemma} \label{lem2}
		\begin{itemize}
			\item [(i)] Let~$W$, $Y$ and~$Z$ be random variables such that~$|W-Y| \leq |Z|$ almost surely.  Then 
			\be{
				\dk(W,Y) \leq  \inf_{\delta >0}\Big(\sup_{s \in \IR}\IP[s<Y \leq s+\delta]+ \IP[|Z| \geq \delta] \Big).
			}
			\item[(ii)] Let~$Y$ be a random variable satisfying 
			\be{
				M_Y:=\sup_{\delta >0} 	\sup_{s \in \IR} \frac{1}{\delta} \IP[s \leq Y\leq s+ \delta ]  < \infty. 
			}
			Then there exists a positive constant~$C=C(M_Y)$, such that for all random variable~$W$, 
			\be{
				\dk(W,Y)\leq C \dw(W,Y)^{1/2}.
			}
		\end{itemize}
	\end{lemma}
  	\begin{proof}
		Since~$Y-|Z| \leq W \leq Y + |Z|$, we have for all~$s\in \IR$ and~$\delta>0$
		\be{
			\IP[Y \leq s- \delta ] - \IP[|Z| \geq \delta]	 \leq \IP[W \leq s] \leq \IP[Y \leq s+ \delta ] + \IP[|Z| \geq \delta].	
		}
		Subtracting~$\IP[Y\leq s]$ everywhere and taking supremum over~$s$, $(i)$ now easily follows.  Item~$(ii)$ is proved by \cite[Proposition 1.2]{R}.
	\end{proof}
  
\subsection{Distributional limit theorem}

\begin{proof}[Proof of Theorem~\ref{thm3}]
We shall prove that for all~$j \in J$ and~$ l \in \IN$,
\ben{ \label{22}
\IE\bclc{|X/n-a_j|^l \given |X-na_j| \leq n\delta_*} =\bigo\bclr{1/(n\sigma_j)^l},
}
and for~$j \in J$
\besn{ \label{23}
		\mel \dw \bclr{\law\clr{W_j\given |X-na_j| \leq n\delta_*}, \law(Y_j)} \\
		&= \bigo(1/(n\sigma_j)) +\bigo(\sigma_*/\sigma_j \I[j \in J\setminus J_*]),
}
where 
\be{
W_j= \sigma_j(X-na_j), \qquad Y_j \propto \bp_j \propto \exp(c_jx^{2m_j}+b_j x),
}
with~$c_j$ and~$b_j$ given as in \eqref{6}. Let~$\tX_j$ be a random variable having the conditional distribution of~$X$ given~$|X-na_j| \leq n\delta_*$; that is,
\ben{ \label{24}
	\pp [\tX_j =k] =\frac{\exp(H_n(k/n))}{Z_{n,j}}, \qquad \ell_j \leq k \leq L_j, 
}
where
\be{
	\ell_j=\lceil n(a_j-\delta_*) \rceil, \quad L_j=  [n(a_j+\delta_*)], \quad Z_{n,j} = Z_{n,j}(\delta_*).
}
Then 
\ben{ \label{25}
\IE\bclc{|X/n-a_j|^l \given |X-na_j| \leq n\delta_*} =\IE\bclc{|\tX/n-a_j|^l},
}
and from  Lemma~\ref{lem1}, we have
\besn{ \label{26}
	\mel\dw \bclr{\law\clr{W_j\given |X-na_j| \leq n\delta_*}, \law(Y_j)}  \\
	&\leq\sup_{f \in C^2_K(\IR)} 
	\bbabs{\IE\bclc{f'(W_j) + \tfrac{\bp'_j(W_j)}{\bp_j(W_j)}f(W_j) \given |X-na_j| \leq n\delta_*}}   \\
	& = \sup_{f \in C^2_K(\IR)} 
	\bbabs{\IE\bclc{f'(\tW_j) + \tfrac{\bp'_j(\tW_j)}{\bp_j(\tW_j)}f(\tW_j)}}
}
where~$K=K(c_j,b_j,m_j)$ is a finite constant, and 
\be{
	\tW_j= 
		\sigma_j(\tX-na_j).
}
Given~$f \in C^2_K(\IR)$, we define the function~$g:\R \rightarrow \R$ as
\be{ 
	g(x) =  			f\left(\sigma_j(x-na_j)  \right).
	}
For any bounded function~$h:\R\to\R$ and~$\delta>0$, let~$\D_\delta h(x)=h(x+\delta)-h(x)$; we have
\ben{ \label{27}
\D_1 g(\tX_j) = \D_{\sigma_j}f(\tW_j), \qquad g(\tX_j) = f(\tW_j).
}
For~$x=(\ell_j-1)/n,\ell_j/n,\ldots,(L_j-1)/n$, let 
\be{
	D_n(x)=\D_{1/n}H_n(x)= n\D_{1/n}A_n(x) +n\sigma_* \D_{1/n} B_n(x),
}
and also let~$D_n((\ell_j-1)/n) = 0$. Note that by \eqref{24}, for~$\ell_j-1\leq k \leq L_j-1$,
\be{
\frac{\pp[\tX_j=k+1]}{\pp[\tX_j=k]} = \exp \left(D_n(k/n)\right).
}
Hence,  \eqref{27} and  straightforward calculations now yield
\ben{\label{28}
\IE\D_{\sigma_j} f(\tW_j) = \IE\D_1 g(\tX_j)  = \IE\bclc{ g(\tX_j)\bcls{\exp \bclr{ -D_n(\tsfrac{\tX_j-1}{n})}-1 }}  + r_1, 
}
where
\be{ 
	r_1= \frac{1}{Z_{n,j}}\bcls{-g(L_j+1)\exp \bclr{ H_n(L_j/n)}+g(\ell_j) \exp\bclr{  H_n(\ell_j/n)} }.
}
By \eqref{14}, we have 
\be{
\max\{H_n(L_j/n), H_n(\ell_j/n) \} \leq H_n(k_j/n) -cn, 
} 
for some~$c>0$. Moreover,  
$$|H_n(k_j/n) - nA_n(k_j/n) -n\sigma_*B(a_j)| = n\sigma_*|B_n(k_j/n)-B(a_j)| =\bigo(n\sigma_*).$$
 Therefore, 
\be{
	\max\{H_n(L_j/n), H_n(\ell_j/n)  \} \leq nA_n(k_j/n)+n\sigma_*B(a_j)- cn/2.
}
Combining this estimate with \eqref{15}, we obtain
\ben{ \label{29}
r_1 \leq \norm{f}_{\infty}\exp(-cn/4).
} 
Moreover, by Taylor's expansion, 
\ben{ \label{30}
	\left|  \frac{1}{\sigma_j}\D_{\sigma_j} f(\tW_j) - f'(\tW_j) \right| \leq  \sigma_j \norm{f''}_{\infty}.
} 
It follows from \eqref{28}, \eqref{29} and \eqref{30}  that 
\bmn{\label{31}
	\left|\IE\bbbclc{  f'(
		\tW_j) - \frac{1}{\sigma_j} 
		\bbclr{\exp \bclr{ - D_n(\tfrac{\tX_j-1}{n}) }-1 } f(\tW_j)} \right|  \\
	  \leq \norm{f}_{\infty}\exp(-cn/4) + \sigma_j\norm{f''}_{\infty}.
}
We now estimate the error when replacing~$\sigma_j^{-1} (\exp ( - D_n(\tfrac{X-1}{n}))-1)$ by $ \bp_j'(W_j)/\bp_j(W_j)$ in \eqref{31}. For~$|k-k_j| \leq \delta_* n$, using~\eqref{A3ii}  and Taylor's expansion we have 
\bes{
	\mel A_n(k/n)-A_n((k-1)/n)-n^{-1}A'(k/n)\\
	&=[A_n(k/n)-A_n((k-1)/n)]-[A(k/n)-A((k-1)/n)]\\
	&\quad + A(k/n)-A((k-1)/n)-n^{-1}A'(k/n) = \bigo(n^{-2}).
}
Thus 
\be{
n \D_{1/n} A_n((k-1)/n) =A'(k/n) +\bigo(n^{-1}).
}
Similarly, 
\be{
n \D_{1/n} B_n((k-1)/n) =B'(k/n) +\bigo(n^{-1}).
}
Therefore,
\ben{ \label{32}
	|D_n((k-1)/n) - [A'(k/n)+\sigma_* B'(k/n)]| =\bigo(n^{-1}).
}
Furthermore,~$|e^u-e^v|=e^v|e^{u-v}-1| \leq 2e^v|u-v|$ when~$|u-v |$ is sufficiently small. Hence, by using \eqref{32} we have for all~$ n$ large enough
\besn{\label{33}
\mel\bigl|\exp ( - D_n((k-1)/n) - \exp(-A'(k/n)-\sigma_* B'(k/n)) \bigr|  \\
&\leq 2 \max_{|x-a_j|\leq \delta_*} \exp(|A'(x)|+\sigma_* |B'(x)|)   \\
& \qquad \times |D_n((k-1)/n) - [A'(k/n)+\sigma_* B'(k/n)]| =\bigo(n^{-1}).
}
Moreover, by applying Taylor's expansion to the function~$e^{-A'(x)-\sigma_*B'(x)}$ around~$x=a_j$ and noting that~$A^{(k)}(a_j)=0$ for all~$1 \leq k \leq 2m_j-1$,   
\bes{
	\exp(-A'(k/n)-\sigma_* B'(k/n))
  & =1-\frac{A^{(2m_j)}(a_j)}{(2m_j-1)!}(k/n-a_j)^{2m_j-1}  -\sigma_* B'(a_j) \\
	&\qquad+\bigo\big((k/n-a_j)^{2m_j} + \sigma_* |k/n-a_j|  \big).
}
Note further that 
\bes{
\mel\frac{A^{(2m_j)}(a_j)}{(2m_j-1)!}(\tX_j/n-a_j)^{2m_j-1}  +\sigma_* B'(a_j) \\
& = 2m_j c_j (\tX_j/n-a_j)^{2m_j-1} +\sigma_* B'(a_j) \\
& = \sigma_j (2m_jc_j\tW_j^{2m_j-1}+b_j)  -\sigma_j b_j +\sigma_*B'(a_j)\\
&= \sigma_j \frac{\bp_j'(\tW_j)}{\bp_j(\tW_j)}  +\bigo(\sigma_* \I[j \in J \setminus J_*]),
}
since~$\bp_j'(w)/ \bp_j(w)=2m_jc_jw^{2m_j-1} +b_j$, and
\be{
\tW_j = \sigma_j(\tX_j-na_j)= \sigma_j^{-1/(2m_j-1)} (\tX_j/n-a_j), 
}
and 
\be{
|\sigma_j b_j -\sigma_*B'(a_j)| =\begin{cases}
	0 &\textrm{if } j \in J_*\\
	|\sigma_*B'(a_j)| = \bigo(\sigma_*)& \textrm{if } j \in J \setminus J_*.
\end{cases}
}
Therefore,
\besn{ \label{34}
	\mel\exp\bigl(-A'(\tX_j/n)-\sigma_* B'(\tX_j/n)\bigr)-1 \\
  & =-\sigma_j \frac{\bp_j'(\tW_j)}{\bp_j(\tW_j)} + \bigo(\sigma_* \I[j \in J\setminus J_*]) \\
  &\quad+\bigo \bigl((\tX_j/n-a_j)^{2m_j} +\sigma_* |\tX_j/n-a_j|  \bigr).
}
It follows from \eqref{33} and \eqref{34}, and the fact that~$\sigma_* \leq \sigma_j$ that 
\besn{ \label{35}
	\mel\IE \bbclc{ \bbabs{ |\sigma_j^{-1}\big(\exp ( - D_n(\tfrac{\tX_j-1}{n})) -1\big) f(\tW_j)+ \tfrac{\bp_j'(\tW_j)}{\bp_j(\tW_j)} f(\tW_j) }} \\
	& \leq  C \norm{f}_{\infty}  \IE\bbclc{\bclr{\sigma_j^{-1}(\tX_j/n-a_j)^{2m_j} + |\tX_j/n-a_j|}}  \\
	&  \quad  + C\sigma_j^{-1}\sigma_* \I[j \in J\setminus J_*], 
}
where~$C$ is a positive constant. In order to estimate the above term, we analyse~$\pp[\tilde{X}_j=k]$. By Proposition~\ref{prop1}, if~$|k/n-a_j| \leq \delta_*$, we have
\besn{ \label{36}
	\mel \IP[\tX_j=k] =\frac{\pp[X_j=k]}{Z_{n,j}(\delta_*)} \\
	  &= \bigo(1) \sigma_j \exp \left(n \left(A_n(k/n) + \sigma_* B_n(k/n) - A_n(k_j/n) - \sigma_* B(a_j) \right) \right).  
}
By \eqref{A3ii},
\be{
	\babs{[A_n \left(k/n\right) - A_n \left(k_j/n\right)] - [A \left(k/n\right) - A \left(k_j/n\right)]} = \bigo_{\delta_*}(|k-k_j|/n^2) = \bigo_{\delta_*}(1/n).
}
Moreover, using Taylor expansion and~\ref{A1} 
\ba{
|A(k/n)-A(k_j/n)| & \leq | A(k/n) - A(a_j)|+ |A(k_j/n)-A(a_j)|  \\ & \leq  \alpha_j(k/n-a_j)^{2m_j} + \bigo(1/n),
}
where 
\be{
	\alpha_j:=\max_{|x-a_j| \leq \delta_*} \frac{ A^{(2m_j)}(x)}{(2m_j)!} <0.
}
Therefore,
\be{
A_n(k/n) \leq A_n(k_j/n) + \alpha_j(k/n-a_j)^{2m_j} + \bigo(1/n).
}
By~\eqref{A3i}, 
\be{
|B_n(k/n)-B(a_j)| \leq |B_n(k/n)-B(k/n)| +|B(k/n)-B(a_j)|=\bigo(|k/n-a_j|).
} 
Using the last two display equations, \eqref{36} and~$ \sigma_* \leq \sigma_j$, we have 
\be{
	\IP[\tX_j=k] \leq C \sigma_j \exp \left( \alpha_j n (k/n-a_j)^{2m_j} + Cn \sigma_j |k/n-a_j| \right)
}
for some finite constant~$C$. Next, by using~$\sigma_j^{2m_j} =n^{1-2m_j}$ and  integral approximations, we have for all~$l \in \IN$, 
\bes{
\mel\E\bclc{|\tX_j/n-a_j|^l} \\
& \leq C \sigma_j \sum_{k: |k/n-a_j| \leq \delta_*} |k/n-a_j|^{l}  \exp \left(\alpha_j n\left(k/n-a_j\right)^{2m_j} + C n \sigma_j   |k/n-a_j| \right)\\
&=\bigo(\sigma_j)  \int_{-n\delta_*}^{n\delta_*} (|x|/n)^l \exp(\alpha_j (x \sigma_j)^{2m_j} + C|x\sigma_j|) dx\\
&=\bigo((n\sigma_j)^{-l} )\int_{-n\sigma_j\delta_*}^{n \sigma_j \delta_*} |y|^l \exp(\alpha_j y^{2m_j}+C|y|) dy= \bigo((n\sigma_j)^{-l}),
}
since~$\alpha_j <0$. This estimate and \eqref{25} implies \eqref{22}. In particular, we  have 
\bes{ 
\mel\IE \left\{ \sigma_j^{-1}(\tX_j/n-a_j)^{2m_j} +|\tX_j/n-a_j| \right\} \\
&=\bigo(\sigma_j^{-1}(n\sigma_j)^{-2m_j}) + \bigo((n\sigma_j)^{-1}) =  \bigo\bclr{(n\sigma_j)^{-1}},
}
where we used that~$\sigma_j^{2m_j} =n^{1-2m_j}$.  Therefore, by \eqref{35},
\be{
 \IE \Bigl\{ \Bigl |\sigma_j^{-1}\big(\exp ( - D_n(\tfrac{\tX_j-1}{n})) -1\big) f(\tW_j)  + \tfrac{\bp_j'(\tW_j)}{\bp_j(\tW_j)} f(\tW_j) \Bigr|  \Bigr\} =\bigo\bclr{\norm{f}_{\infty}/(n\sigma_j)}.
}
Combining the above inequality with  \eqref{31} we yield that for all~$K>0$
\bes{ 
	\mel\sup_{f \in C^2_K(\IR)}   \bbabs{\IE \bbclc{f'(\tW_j) + \tfrac{\bp_j'(\tW_j)}{\bp_j(\tW_j)}f(\tW_j)}} \\
	&= \bigo(K/(n\sigma_j)) +\bigo(\sigma_*) +\bigo(\sigma_j/\sigma_* \I[j \in J\setminus J_*])\\
	& =\bigo(K/(n\sigma_j))+\bigo(\sigma_j/\sigma_* \I[j \in J\setminus J_*]).
}
Then the desired estimate  \eqref{23}  follows from this bound and \eqref{26}. 
\end{proof}

\subsection{Free energy }
\begin{proof}[Proof of Proposition~\ref{prop1}]
 Fix a constant~$\delta \in (0, \delta_*]$.  We aim to approximate 
	\be{
	Z_{n,j}(\delta):= \sum_{|k/n-a_j| \leq \delta} \exp(H_n(k/n)).	
}
		Let~$\eps \in (0,\delta)$ be a  suitably small constant chosen later (see \eqref{39}). For~$n\eps\leq |k-na_j| \leq n\delta$,  by~\eqref{A3ii}
\bes{
	\mel A_n(k/n) -A_n(k_j/n) = A(k/n) - A(k_j/n) + \bigo(1/n) \\
	&\leq \max_{\eps \leq |x-a_j|\leq \delta} (A(x)-A(a_j)) + \bigo(|k_j/n-a_j|) + \bigo(1/n)
	\leq -\eta
}
with~$\eta=\eta(\eps)>0$, since~$a_j$ is the unique maximizer of the smooth function~$A$ in~$[a_j-\delta_*,a_j+\delta_*]$. Therefore, since~$B_n$ is uniformly bounded by \eqref{13},
\bes{
	\mel H_n(k/n) - H_n(k_j/n)\\
	&= n[A_n(k/n) - A_n(k_j/n)] + n \sigma_*  [B_n(k/n) - B_n(k_j/n)] \leq -\eta n/2.
}
Thus 
\ben{ \label{37}
	\sum_{k=0}^n \frac{\exp(H_n(k/n))}{\exp(H_n(k_j/n))} \I[ \eps \leq |k/n-a_j| \leq \delta] \leq n \exp(-\eta n/2).
}
Next, we consider~$\sigma_j^{-1} \log n\leq |k-na_j| \leq n\eps$. By~\ref{A3} for all~$|k/n-a_j| \leq \delta_*$
\be{
A_n(k/n) -A_n(k_j/n) = A(k/n) - A(k_j/n) + \bigo(1/n).
}
Moreover, using Taylor expansion around~$a_j$ with~$A^{(m)}(a_j)=0$ for~$1 \leq m \leq 2m_j-1$, we have
\bes{
	A(k/n)-A(k_j/n)
	& = A(k/n) - A(a_j)+ A(a_j)- A(k_j/n) \\
	& = c_j(k/n-a_j)^{2m_j} + \bigo(|k/n-a_j|^{2m_j+1}) + \bigo(n^{-2}), 
}
where we recall that~$c_j=A^{(2m_j)}(a_j)/(2m_j)!$ and~$|k_j/n-a_j|^2 \leq n^{-2}$. It follows from the last two estimates that for all~$|k-na_j| \leq n \delta_*$
\besn{ \label{38}
  \mel A_n(k/n) -A_n(k_j/n) \\
  &= c_j(k/n-a_j)^{2m_j} + \bigo(|k/n-a_j|^{2m_j+1}) 
   + \bigo(n^{-1}).
}
In particular,   there exists a constant~$C_1=C_1(a_j,c_j,A)>0$ such that 
\be{
	A_n(k/n) -A_n(k_j/n) \leq c_j (k/n-a_j)^{2m_j} + C_1|k/n-a_j|^{2m_j+1} + C_1/n. 
}
By taking 
\ben{ \label{39}
\varepsilon = |c_j|/(2C_1),
}
we yield that for~$|k/n-a_j| \leq \varepsilon$,
\ben{ \label{40}
A_n(k/n) -A_n(k_j/n) \leq c_j (k/n-a_j)^{2m_j}/2 +  C_1/n,
}
by noting that~$c_j<0$.  On the other hand for all~$|k/n-a_j| \leq \delta_*$, by~\ref{A3}
\be{ 
	n\sigma_*  [B_n(k/n)-B_n(k_j/n)] 
  = n \sigma_* [B(k/n)-B(k_j/n)] +\bigo(\sigma_*).
}
Moreover,
\bes{
B(k/n)-B(k_j/n)&=B(k/n)-B(a_j)+B(a_j)-B(k_j/n) \\
&= B'(a_j) (k/n-a_j) + \bigo(|k/n-a_j|^2) + \bigo(n^{-1}).
}
Thus for  all~$|k/n-a_j| \leq \delta_*$,
\besn{ \label{41}
 \mel n\sigma_*  [B_n(k/n)-B_n(k_j/n)]\\
 & = \sigma_*(k-na_j) (B'(a_j)+\bigo(|k/n-a_j|)) +\bigo(\sigma_*).
}
Hence, using \eqref{40} and \eqref{41} and~$\sigma_* \leq \sigma_j$, and noting that~$\sigma_j^{2m_j}=n^{1-2m_j}$,
\be{
	H_n(k/n) - H_n(k_j/n) \leq\frac{c_j}{2}(\sigma_j(k-na_j))^{2m_j} + C\sigma_j |k-na_j| + C,
}
with~$C$ some positive constant. Therefore,
\besn{ \label{42}
	\mel\sum_{k=0}^n  \frac{\exp(H_n(k/n))}{\exp(H_n(k_j/n))} \I[ (\log n)/\sigma_j \leq |k-na_j| \leq n \eps]  \\
	&\leq \sum_{|k-na_j|\geq (\log n)/\sigma_j} \exp \left( \frac{c_j}{2}(\sigma_j(k-na_j))^{2m_j} + C\sigma_j |k-na_j| +C \right)  \\
	&=\bigo(1) \int_{|x| \geq (\log n)/\sigma_j } \exp\left(\frac{c_j}{2}(\sigma_j x)^{2m_j} + C|\sigma_j x| + C \right)dx 
  = \bigo(1/n).
}
Here, in the last inequality we have used~$\int_{|y| \geq \log n} \exp(c_jy^{2m_j}+Cy +C) dy = \bigo(n^{-2})$ since~$c_j<0$ and~$m_j \geq 1$. It follows from \eqref{37} and \eqref{42} that
\besn{ \label{43}
	Z_{n,j} (\delta)&= \big(1+\bigo(1/n)\big) \sum_{\substack{|k-na_j|\phantom{AA}\\ \leq (\log n)/\sigma_j}} \exp(H_n(k/n))  \\
	&= \big(1+\bigo(1/n)\big) \exp(H_n(k_j/n)) \sum_{\substack{|k-na_j|\phantom{AA}\\\leq (\log n)/\sigma_j}} \frac{\exp(H_n(k/n))}{\exp(H_n(k_j/n))}  \\
	&= \big(1+\bigo(\sigma_*)\big) \exp(nA_n(k_j/n)+n\sigma_* B(a_j)) \\
	&\hspace{2.5 cm} \times \sum_{\substack{|k-na_j|\phantom{AA}\\\leq  (\log n)/\sigma_j}} \frac{\exp(H_n(k/n))}{\exp(H_n(k_j/n))},
}
where for the last equation we used~\ref{A3} to derive that
\be{
	\big| nA_n(k_j/n)+n\sigma_* B(a_j) - H_n(k_j/n) \big| = \big| n\sigma_*\big(B_n(k_j/n) -B(a_j) \big) \big| =\bigo(\sigma_*).
}
By  \eqref{41}, if~$|k-na_j|\leq \sigma^{-1}_j \log n$ then
\bes{
	\mel n \sigma_*  [B_n(k/n)-B_n(k_j/n)]\\
	&= B'(a_j)\sigma_*(k-na_j) +\bigo((\log n)^2\sigma_*/n\sigma^2_j) +\bigo(\sigma_*)\\
	&=b_j\sigma_j(k-na_j) + \bigo(\sigma_* (\log n)/\sigma_j\I[\sigma_j  \neq \sigma_*]) \\
	& \quad +\bigo((\log n)^2\sigma_*/n\sigma^2_j) +\bigo(\sigma_*),
}	
since~$b_j=B'(a_j)\I[\sigma_j = \sigma_*]$. Similarly,  by \eqref{38} for~$|k-na_j|\leq \sigma^{-1}_j \log n$,
\bes{	
	n[A_n(k/n) -A_n(k_j/n)] &=c_jn(k/n-a_j)^{2m_j}+\bigo((\log n)^{2m_j+1}/n\sigma_j)\\
	&=c_j(\sigma_j(k-na_j))^{2m_j}+\bigo((\log n)^{2m_j+1}/n\sigma_j).
}
Therefore,
\ben{ \label{44}
H_n(k/n)-H_n(k_j/n) = c_j(\sigma_j(k-na_j))^{2m_j}+ b_j\sigma_j(k-na_j) + \bigo(\tau_j),
}
where 
\be{
 \tau_j =\frac{(\log n)^{2m_j+1}}{n\sigma_j}  + \frac{\sigma_* \log n}{\sigma_j}\I[\sigma_j \neq \sigma_*].
}
We now compute 
\besn{ \label{45}
\mel\sum_{\substack{|k-na_j|\phantom{AA}\\ \leq (\log n)/\sigma_j}} \exp\left(c_j(\sigma_j(k-na_j))^{2m_j} + b_j\sigma_j(k-na_j)\right)\\
&= \sum_{i \in \Gamma_n} \exp \left(c_j(i\sigma_j)^{2m_j} + b_j(i\sigma_j)\right),
}
where~$\Gamma_n=\{k-na_j: k\in \mathbb{Z}, \, |k-na_j| \leq  (\log n)/\sigma_j \}$. Denote by~$h(x)=\exp \left(c_jx^{2m_j} + b_jx\right)$. Then for all~$i \in \Gamma_n$, by Taylor expansion
$$\Big|h(i\sigma_j)-\sigma_j^{-1} \int_{i\sigma_j}^{(i+1) \sigma_j} h(x) dx \Big| \leq \sigma_j \sup_{i \sigma_j \leq x \leq (i+1)\sigma_j}|h'(x)|.~$$
Hence,
\besn{ \label{46}
\mel \Bigl|\, \sum_{i \in \Gamma_n}h(i\sigma_j)-\sigma_j^{-1} \int_{\R} h(x) dx \Bigr|\\
	&\leq\sigma_j \sum_{i \in \Gamma_n}\sup_{i \sigma_j \leq x \leq (i+1)\sigma_j}|h'(x)| + \int_{|x| \geq \log n} h(x) dx. 
}
Since~$h'(x)=\exp(c_jx^{2m_j}+b_jx)(2m_jc_jx^{2m_j-1}+b_j)$ with~$c_j<0$, we can find a positive constant~$C=C(c_j,m_j,b_j)$, such that if~$|y|\geq C$ then 
\be{
\sup_{x \in \R} |h'(x)| \leq C, \qquad  \sup_{y \leq x \leq y+1} |h'(x)| \leq \exp(-c_jy^{2m_j}/2).
}
Therefore, we have 
\bes{
\sum_{i \in \Gamma_n}\sup_{i \sigma_j \leq x \leq (i+1)\sigma_j}|h'(x)| &\leq 2C^2/\sigma_j + \sum_{i \in \Gamma_n} \exp(-c_j(i\sigma_j)^{2m_j}/2)\\
& \leq \bigo(1/\sigma_j) + \int_{|x| \leq (\log n)/\sigma_j} \exp(-c_j(x\sigma_j)^{2m_j}/2) dx\\
&=\bigo(1/\sigma_j),
}
which together  with \eqref{46} yields that 
\bes{
 \sum_{i \in \Gamma_n} h(i\sigma_j) &=  \sigma_j^{-1} \int_{\R} h(x) dx + \bigo(1) + \int_{|x|\geq \log n} h(x) dx\\
 &= \sigma_j^{-1} q_j + \bigo(1),
}
since~$q_j=\int_{\R}h(x) dx$. Combining this with \eqref{44} and \eqref{45} we obtain that
\bes{
	\mel\sum_{\substack{|k-na_j|\phantom{AA}\\ \leq (\log n)/\sigma_j}}\frac{\exp(H_n(k/n))}{\exp(H_n(k_j/n))} 
	= (1+\bigo(\tau_j)) \sigma_j^{-1} q_j   +\bigo(1)
=(1+\bigo(\tau_j)) \sigma_j^{-1} q_j,
}
since~$\tau_j \geq (\log n)^{2m_j+1}/(n\sigma_j) \geq \sigma_j$. 
We finally deduce  \eqref{15} from the above estimate and \eqref{43}.
\end{proof}

\section{maximum likelihood estimator of linear models} \label{sec2}
We first recall the generalized linear model \eqref{3} given as 
\be{ 
	\mu_n(\omega) =\frac{1}{Z_n} \exp (H_n(\omega)), \qquad \omega \in \Omega_n= \{+1,-1\}^n,
}
where 
\be{
Z_n =\sum_{\omega \in \Omega_n} \exp(H_n(\omega)),
}
and
\be{
H_n(\omega) = n\bclr{\beta_1f_1(\-\omega_+)+\ldots+\beta_l f_l(\-\omega_+)}, \quad \-\omega_+ =\frac{\babs{ \{i: \omega_i=1\}}}{n}. 
}
Since we construct the estimator for each parameter~$\beta_i$ considering the others~$(\beta_j)_{j\neq i}$ to be known, for simplicity we rewrite 
\ben{ \label{47}
H_n(\omega)=n\bclr{\beta f(\-\omega_+)+g(\-\omega_+)},
}
where $f, g: [0, 1] \rightarrow \IR$ are non-constant smooth enough and known functions. Our aim is to estimate the parameter~$\beta$.  In order to build the MLE of~$\beta$, we compute the log-likelihood function of the model as 
\be{
	L_n(\beta, \omega) =\frac{1}{n} \log \mu_n(\omega) = \beta f(\-\omega_+)+ g(\-\omega_+) - \phi_n(\beta)
}
with 
\be{
	\phi_n(\beta) = \frac{1}{n}\log Z_n.
}
Then the MLE of~$\beta$, denoted by~$\hbe_n$,  is a solution of  
\be{
	0=\partial_{\beta} L_n = f(\-\omega_+) - u(\beta), 
}
where 
\be{
	u(\beta)  = \partial_{\beta} \varphi_n  = \IE_{\beta} f(\-\omega_+)
}
with~$\IE_{\beta}$ the Gibbs expectation with respect to~$\mu_n$ for given~$\beta$. Note that 
\be{
	\partial_{\beta} u =\IE_{\beta}  f(\-\omega_+)^2 -\IE_{\beta} \clc{f(\-\omega_+)}^2 >0
}
since~$f$ is non-constant. Therefore, $u$ is strictly increasing in~$\beta$, and thus
\ben{ \label{48}
	\hbe_{n} = u^{-1}\left(f(\-\omega_+)\right).
}
Before stating the main result of this section, recall the entropy function~$I: [0,1] \rightarrow \R$ defined as~$I(a)=-a\log a+(a-1) \log (1-a)$ for~$a \in [0,1]$ with the convention that~$0\cdot\log 0 =0$.
\begin{theorem} \label{thm4} Consider the maximum likelihood estimator~$\hbe_n$ as in \eqref{48} of the linear model having Hamiltonian given by \eqref{47} with~$f, g  \in C^{2m_*+1}([0,1])$ and~$m_*\in \IN$.  Suppose that the function~$A:[0,1] \rightarrow \IR$ given as~$A(a)=\beta f(a)+g(a) +I(a)$ has finite maximizers, denoted by~$(a_j)_{j \in J}$, satisfying that~$A^{(k)}(a_j)= 0$ for all~$1 \leq k \leq 2m_j-1$ and~$A^{(2m_j)}(a_j)<0$ for all~$j \in J$, with~$(m_j)_{j \in J} \subset \N$ and~$m_*=\max_{j \in J} m_j$. Define 
\ba{
	J^+_1&= \{j \in J: f(a_j) = \max_{k \in J} f(a_k) \}, \quad 	J_2^+ = \{ j \in J_1^+: m_j  = \max_{k \in J^+_1} m_k \},\\
	J^-_1&= \{j \in J: f(a_j) = \min_{k \in J} f(a_k) \}, \quad 	J_2^- = \{ j \in J^-_1: m_j = \max_{k \in J_1^-} m_k\}.
} 	
Assume that~$(J_2^- \cup J_2^+) \subset J_* :=\{j \in J: m_j = m_*\}$, and assume that there exist $j \in J_2^-$ and $k \in J_2^+$ such that 
\ben{\label{49}
  f'(a_j)f'(a_k) \neq 0.
}
Then 
\be{
(\hbe_n-\beta)n^{1-1/(2m_*)} \stackrel{\law}{\longto} U,
}
where the distribution of~$U$ is given as in \eqref{66a}--\eqref{66c}.
\end{theorem}
\begin{proof} For simplicity we omit the subscript~$n$ in all involved terms. Let 
	\be{
X=n\-\omega_+, \qquad \text{$\sigma_j = n^{1/(2m_j)-1}$ for~$j \in J$,} \qquad \sigma_*= n^{1/(2m_*)-1}.
}
For~$\gamma \in \IR$, we call~$\pp_{\gamma}$ the Gibbs measure at parameter~$\gamma$ and~$\IE_{\gamma}$ the corresponding expectation.  With~$X=n\-\omega_+$, we have for~$0\leq k \leq n$ that
\be{
\pp_{\beta}[X=k] \propto \exp\left(n\bclr{\beta f(k/n) +g(k/n)}\right) \binom{n}{k} = \exp \left( n A_n(k/n)\right),
}
where~$A_n: \{0, 1/n, \ldots, 1 \} \rightarrow \R$ is defined as
\be{
A_n(k/n) = \beta f(k/n) + g(k/n) + \frac{1}{n} \log \binom{n}{k}.
}
Recall that~$A(a)=\beta f(a)+g(a)+I(a)$.   Let~$B\in C^2([0,1])$ and define~$B_n : \{0, 1/n, \ldots, 1\} \rightarrow \IR$  as $B_n(k/n)=B(k/n)$ for~$0 \leq k \leq n$. As shown in Remark \ref{rem:a4},   there exist positive constants $\eps_*$, $\delta_*$, $C_*$ and real numbers $(\nu_j)_{j \in J_2}$ given in \eqref{nuj}  such that \ref{A1}--\ref{A4} hold. For any~$j \in J$, we define the event
\be{
\kA_j =\{|X/n-a_j| \leq \delta_* \},
}		 
and for~$t \in \R$ define the random variable
\be{
Y_j(t) \propto \exp(c_jx^{2m_j}+tb_jx),
}  
where
\be{
	c_j =\frac{A^{(2m_j)}(a_j)}{(2m_j)!}, \qquad b_j = B'(a_j) \I[j \in J_*]. 
}
 Fix~$t<0$, by the definition of~$\hbe$ and the monotonicity of~$u$ we have 
\bes{
	\mel\pp_{\beta}[(\hbe-\beta)/\sigma_{*} \leq t]	\\
  & = \pp_{\beta} \left[ u^{-1}(f(X/n)) \leq \beta+t \sigma_* \right] =\pp_{\beta}[f(X/n) \leq u(\beta +t\sigma_*)].
}
\paragraph{Part 1.} We start by estimating~$u(\beta+t\sigma_*)$. Note that~$u(\beta+t\sigma_*)=\E_{\beta+t\sigma_*} f(X/n) $, and in the application of Theorem~\ref{thm1}, the measure~$\pp_{\beta + t\sigma_*}$ corresponds to the case~$B=tf$. Hence, with~$t<0$, we have~$J_1\equiv J_1^-$ and~$J_2\equiv J_2^-$. Thus by Theorem~\ref{thm1},
\be{
X/n \stackrel{\pp_{\beta + t\sigma_*}}{\longto} \sum_{j \in J_2^-} p_j^-(t) \delta_{a_j},
}
where for~$j,i  \in J_2^-$, we have
\be{
p_j^-(t) =\frac{q_j(t) e^{\nu_j}}{\sum_{i \in J_2^-}q_i(t) e^{\nu_i}}, \qquad q_i(t) =\int_{\IR} \exp(c_ix^{2m_i}+tb_ix) dx, 
}
and recall from \eqref{nuj} that
\[
\nu_j= \log \sqrt{\frac{1}{(1-a_j)a_j}}.
\]
Note that~$b_j=B'(a_j)\I[j \in J_* ] =B'(a_j)$ for~$j \in J_2^-$, since  we assume that~$J_2^- \subset J_*$. This assumption also yields that~$\sigma_{j_2}=\sigma_*$ for all~$j_2 \in J_2^-$. Therefore,  using Theorem~\ref{thm2}, we have
\bgn{ \label{50a}
\pp_{\beta+t\sigma_*}[\kA_j] = p_j^-(t) + \bigo(\tau_*+\tau^-_* )\quad  \text{for all~$j\in J_2^-$,}\\
 \label{50b}
\pp_{\beta+t\sigma_*}[\kA_j] = \bigo(\sigma_*/\sigma_j) \quad  \text{for all~$j\in J_1^-\setminus J_2^-$,}\\
 \label{50c}
\pp_{\beta+t\sigma_*}\bcls{\cap_{j\in J_1^-}\kA_j^c} \leq \exp(-cn\sigma_*),
}
where~$c$ is a positive constant and 
\be{
\tau_*=(\log n)^{2m_*+1}/(n\sigma_*), \quad \tau^-_* = \max_{j \in J_1^-\setminus J_2^-} \sigma_*/\sigma_j.
}
 In addition,  Theorem~\ref{thm3} yields that for any~$j \in J$,
\ben{ \label{51}
	\E_{\beta+t\sigma_j}\bclc{(X/n-a_j)^2\given \kA_j} =\bigo(1/(n\sigma_j)^2),
}
and
\besn{ \label{52}
	\mel\dw \bclr{\law_{\pp_{\beta + t\sigma_*}}\clr{\sigma_j(X-na_j)\given \kA_j}, \law(Y_j(t))}\\ &=\bigo(1/(n\sigma_j)) + \bigo(\sigma_*/\sigma_j \I[j \not \in J_*]).
}
  We remark that here and below the notation~$\bigo$   depends on~$\norm{B}_{\infty}=|t|\norm{f}_{\infty}$ and~$\norm{A}_{\infty}$. 
Let $\lambda_- =\min_{j \in J} f(a_j)$. Then $\lambda_-=f(a_j)$ for all~$j \in J_1^-$, 
and therefore
\besn{ \label{53}
u(\beta+t\sigma_*) - \lambda_- &= \E_{\beta+t\sigma_*}\bclc{f(X/n)-\lambda_-} \\
&= \sum_{j \in J_1^-}  \E_{\beta+t\sigma_*}\bclc{f(X/n)-f(a_j)\given \kA_j} \pp_{\beta + t\sigma_*}[\kA_j] \\
& \qquad \quad + \E_{\beta+t\sigma_*} \bclc{ (f(X/n)-\lambda_-)\I\bcls{\cap_{j \in J_1^-} \kA_j^c}}.
}
For~$j \in J$, by Taylor's expansion, 
\bes{
\mel\E_{\beta+t\sigma_*}\clc{f(X/n)-f(a_j)\given \kA_j} \\
& =  \E_{\beta+t\sigma_*}\bclc{ f'(a_j) \sigma_k(X-na_j)\given \kA_j}/(n\sigma_j)
 + \bigo(1)\E_{\beta+t\sigma_*}\bclc{(X/n-a_j)^2 \given \kA_j}.
}
In addition, by \eqref{52},
\bes{
\mel\E_{\beta+t\sigma_*}\bclc{ f'(a_j) \sigma_j(X-na_j)\given \kA_j} \\
& =f'(a_j)\E Y_j(t) 
  + \bigo(1/(n\sigma_j)) + \bigo(\sigma_*/\sigma_j \I[j \not \in J_*]).
}
The last two estimates and \eqref{51} yields that  
\besn{ \label{54}
\mel \E_{\beta+t\sigma_*}\clc{f(X/n)-f(a_j)\given \kA_j} \\
&= f'(a_j) \E Y_j(t) /(n\sigma_j)
+ \bigo(1/(n\sigma_j)^2) +\bigo(\sigma_*/n \sigma^2_j \I[j \not \in J_*]).
}
Combining this with \eqref{50a} and the fact that~$\sigma_j=\sigma_*$ for all~$j\in J_2^-$, and~$J_2^-\subset J_*$, we obtain that
\bes{
\mel\sum_{j \in J_2^-}  \E_{\beta+t\sigma_*}\clc{f(X/n)-f(a_j)\given \kA_j} \pp_{\beta + t\sigma_*}[\kA_j]\\
& =(n\sigma_{*})^{-1} \sum_{k \in J_2^-} f'(a_j) \E Y_j(t) p_j^-(t) + \bigo((\tau_*+\tau^-_*)/n\sigma_*).
}
Using \eqref{50b} and \eqref{54},  we have
\bes{
\mel\sum_{j \in J_1^-\setminus J_2^-} \E_{\beta+t\sigma_*}\clc{f(X/n)-f(a_j)\given \kA_j} \pp_{\beta + t\sigma_*}[\kA_j] \\
&= \bigo(1)\sum_{j \in J_1^-\setminus J_2^-} \sigma_{*}/n\sigma_j^2 
 = \bigo(\tau_*^-/n\sigma_*),
}
and by \eqref{50c} 
\be{
\E_{\beta+t\sigma_*} \bclc{ (f(X/n)-\lambda_-)\I\bcls{\cap_{j \in J_1^-} \kA_j^c}} \leq \exp(-cn\sigma_*/2).
}
It follows from  the last three display equations and  \eqref{53}  that 
\ben{ \label{55}
n\sigma_*(u(\beta+t\sigma_*)- \lambda_-) =  e_-(t)  +\bigo(\tau_*+\tau^-_*),
}
where
\ben{ \label{56}
e_-(t)=\sum_{j \in J_2^-} f'(a_j) \E Y_j(t) p_j^-(t). 
}
Note that 
\be{
e_-(t) =\frac{\sum_{j \in J_2^-} \int_{\IR} f'(a_j) x \exp(c_jx^{2m_j}+tf'(a_j)x) dx }{\sum_{j \in J_2^-} \int_{\IR}  \exp(c_jx^{2m_j}+tf'(a_j)x) dx }.
}
Moreover,  if~$f'(a_j) \neq 0$ by changing variable~$y=tf'(a_j)x$,
\bes{
\mel\int_{\IR} f'(a_j) x \exp(c_jx^{2m_j}+tf'(a_j)x) dx\\
& = \frac{\mathop{\mathrm{sgn}}(t f'(a_j))}{t^2 f'(a_j)} \int_{\IR}y  \exp\left(c_j y^{2m_j}/(tf'(a_j))^{2m_j}+y\right) dy  <0,
}
since~$t<0$ and~$\int_{\IR} y\exp(cy^{2m}+y)dy>0$ for all~$c<0$ and~$m \in \IN$.  In addition, by the assumption \eqref{49}  there exists~$j \in J_2^-$ such that~$f'(a_j) \neq 0$. Thus by the two above display equations, we have 
\be{
e_-(t) \in (-\infty, 0)
}
is a negative and finite constant. 

\paragraph{Part 2.} We proceed to compute~$\pp_{\beta}[f(X/n) \leq u(\beta+t\sigma_*)]$. In the application of Theorem~\ref{thm1}, the measure~$\pp_{\beta}$ corresponds to the case~$B\equiv 0$, or~$J_1=J$ and~$J_2=J_*$. Hence, by Theorem~\ref{thm1}, we have
\be{
	X/n \stackrel{\pp_{\beta}}{\longto} \sum_{j \in J_*} p_j \delta_{a_j},
}
where for~$i$ and~$j$ in~$J$,
\be{
	p_j = \frac{q_j e^{\nu_j} }{\sum_{i \in J_*} q_i e^{\nu_i}}, \quad q_i= \int_{\IR} \exp(c_ix^{2m_i}) dx, \quad \nu_i = \log \sqrt{\frac{1}{(1-a_i)a_i}}.
}
Moreover, by Theorem~\ref{thm2},
\ben{ \label{57}
	\pp_{\beta}[\kA_j] = p_j +\bigo(\tau_*+ \tau_*') \quad \text{for all~$j \in J_*$,}
	\qquad \pp_{\beta}\left[\cap_{j\in J_*} \kA_j^c\right]  =\bigo(\tau'_*),
}
where~$\tau'_*=\max_{j \in J \setminus J_*} \sigma_*/\sigma_j$. By Theorem~\ref{thm3}, 
\ben{ \label{58}
\E_{\beta}\bclc{(X/n-a_j)^2 \given \kA_j}=\bigo\bclr{(n\sigma_j)^{-2}},
}
and
\ben{ \label{59}
\dw\bclr{\law_{\pp_{\beta}}\clr{\sigma_j(X-na_j)\given \kA_j}, \law(Y_j)} =\bigo\bclr{(n\sigma_j)^{-1}},
}
where~$Y_j=Y_j(0) \propto \exp(c_jx^{2m_j})$. It follows from \eqref{57} that
\besn{  \label{60}
\mel\pp_{\beta}[f(X/n) \leq u(\beta+t\sigma_*)] \\
&= \sum_{j \in J_*} \pp_{\beta}[f(X/n) \leq u(\beta+t\sigma_*) \given \kA_j] p_j
+ \bigo(\tau_*+\tau'_*).
}
By Lemma~\ref{lem2} (ii) and \eqref{59}
\besn{ \label{61}
	\dk(\law_{\pp_{\beta}}\clr{\sigma_j(X-na_j)\given \kA_j}, \law(Y_j))  &\leq \dw(\law_{\pp_{\beta}}\clr{\sigma_j(X-na_j)\given \kA_j}, \law(Y_j)) ^{1/2} \\
	& = \bigo((n\sigma_j)^{-1/2}).
}
In particular, for all~$\delta >0$
\bes{
\mel\sup_{s \in \IR}\pp_{\beta }\bcls{s \leq f'(a_j)\sigma_j(X-na_j) \leq s+\delta\given A_j} \\
& \leq \sup_{s \in \IR} \pp_{\beta }\bcls{s \leq f'(a_j)Y_j \leq s+\delta} + \bigo((n\sigma_j)^{-1/2}) = \bigo(\delta) + \bigo((n\sigma_j)^{-1/2}),
}
since~$Y_j$ has the bounded density. Using  the inequality that~$|f(x)-f(a)-  f'(a)(x-a)| \leq \norm{f}_{\infty}(x-a)^2/2$ and Lemma~\ref{lem2}$(i)$, and the above estimate, we have
\bes{
	\mel\dk \bclr{\law_{\pp_{\beta }}\cls{n\sigma_j(f(X/n)-f(a_j))\given \kA_j}, \law(f'(a_j)\sigma_j(X-na_j)\given \kA_j)}  \\
	&\leq \inf_{\delta>0} \Big(\sup_{s \in \IR}\pp_{\beta }[s \leq f'(a_j)\sigma_j(X-na_j) \leq s+\delta\given \kA_j] \\
	& \qquad \qquad +\pp_{\beta}[\norm{f''}_{\infty}(n\sigma_j(X/n-a_j)^2) \geq 2 \delta \given \kA_j]\Big) \\
	& = \bigo(1)\inf_{\delta >0} \bclc{ \delta + \pp_{\beta}\bcls{\norm{f''}_{\infty}(n\sigma_j(X/n-a_j)^2) \geq 2 \delta \given \kA_j} } + \bigo\bclr{(n\sigma_j)^{-1/2}}.
}
Moreover, by Markov's inequality and \eqref{58}
\bes{
	\mel\pp_{\beta}\bcls{\norm{f''}_{\infty}(n\sigma_j(X/n-a_j)^2) \geq 2 \delta \given \kA_j} 	\\
	& =\bigo(1)  \IE\bclc{n\sigma_j(X/n-a_j)^2 \given \kA_j}/\delta
	=\bigo\bclr{(\delta n\sigma_j)^{-1}}.
}
Combining the two above estimates and taking~$\delta=(n\sigma_j)^{-1/2}$, we obtain
\bm{
\dk \bclr{\law_{\pp_{\beta }}\clr{n\sigma_j(f(X/n)-f(a_j))\given \kA_j}, \law\clr{f'(a_j)\sigma_j(X-na_j)\given \kA_j}}\\
 = \bigo((n\sigma_j)^{-1/2}),
}
which together with  \eqref{61} implies that   for all~$j \in J$
\ben{ \label{62}
	\dk \bclr{\law_{\pp_{\beta }}\clr{n\sigma_j(f(X/n)-f(a_j))\given \kA_j}, \law(f'(a_j)Y_j)} =\bigo((n\sigma_j)^{-1/2}).
}
If~$j \in J_* \setminus J_1^-$  then  by the definition of~$J_1^-$, we have  $f(a_j) > \lambda_-$. Hence, by~\eqref{55}, 
\be{
u(\beta +t\sigma_{*}) =\lambda_-+o(1) \leq (f(a_j) + \lambda_-)/2.
}
Thus
\bes{
	\mel\pp_{\beta}[f(X/n) \leq u(\beta+t\sigma_*) \given \kA_j]  \\
	& \leq   \pp_{\beta}[f(X/n) \leq (\lambda_- + f(a_j))/2  \given \kA_j]  \\
	& = \pp_{\beta}[ n\sigma_j(f(X/n)-f(a_j)) \leq n\sigma_j(\lambda_--f(a_j))/2  \given \kA_j]  \\
	& \leq \dk \bclr{\law_{\pp_{\beta }}\clr{n\sigma_j(f(X/n)-f(a_j))\given \kA_j}, \law(f'(a_j)Y_j)}  \\
	& \quad +\pp[f'(a_j)Y_j \leq n\sigma_j(\lambda_--f(a_j))/2]   
	= \bigo((n\sigma_j)^{-1/2}), 
}
by using \eqref{62} and the following estimate 
\be{
\pp[f'(a_j)Y_j \leq n\sigma_j(\lambda_--f(a_j))/4] \leq \exp(-c(n\sigma_j)^2),
}
for some~$c>0$, since~$Y_j \propto \exp(c_jx^{2m_j}+b_jx)$ with~$c_j<0$, and~$\lambda_- < f(a_j)$. 

Next, assume that~$j \in J_* \cap J_1^-$.  Then~$\sigma_j=\sigma_*$ and~$f(a_j)=\lambda_-$. Therefore,
\bes{
  \mel\pp_{\beta}[f(X/n) \leq u(\beta+t\sigma_*) \given \kA_j] \\
& = \pp_{\beta}\bcls{n\sigma_j(f(X/n)-f(a_j)) \leq n\sigma_*(u(\beta+t\sigma_*) -\lambda_-) \given \kA_j}.
}
Combining this with \eqref{62} yields that 
\besn{ \label{63}
\mel\pp_{\beta}\bcls{f(X/n) \leq u(\beta+t\sigma_*) \given \kA_j}\\ &=  \pp \bigl[f'(a_j) Y_j \leq n\sigma_*(u(\beta+t\sigma_*) -\lambda_-)\bigr]  + \bigo(1/(n\sigma_*)^{1/2}).
}
Recall that by \eqref{55}
\be{
n\sigma_*(u(\beta +t\sigma_*)- \lambda_-) =  e_-(t) + \bigo(\tau_*+\tau_*^-), 
}
where~$ e_-(t) \in (-\infty, 0)$ is given in \eqref{56}. 
Hence, if~$f'(a_j)=0$ then 
\ben{ \label{64}
\pp[f'(a_j) Y_j \leq n\sigma_*(u(\beta+t\sigma_*) -\lambda_-)] =0.
}
 If~$f'(a_j) \neq 0$, since~$Y_j$ has the symmetric law with bounded density,
\bes{
\mel\pp[f'(a_j) Y_j \leq n\sigma_*(u(\beta+t\sigma_*) -\lambda_-)]\\ &=\pp\left[ Y_j \leq n\sigma_*(u(\beta+t\sigma_*) -\lambda_-)/f'(a_j)\right]\\
&=\pp \left[Y_j \leq e_-(t)/f'(a_j) +\bigo(\tau_*+\tau^-_*) \right]\\
& =\pp \left[Y_j \leq e_-(t)/f'(a_j)  \right] +\bigo(\tau_*+\tau^-_*).
}
Combining this with \eqref{63}, we obtain that if~$f'(a_j) \neq 0$ then
\besn{ \label{65}
	\mel\pp_{\beta}[f(X/n) \leq u(\beta+t\sigma_*) \given \kA_j] \\ 
  &= \pp \left[Y_j \leq e_-(t)/f'(a_j)  \right] 
	  +\bigo(\tau_*+\tau_*^-)+ \bigo(1/(n\sigma_{*})^{1/2}). 
}

\paragraph{Part 3.} We now combine the results from Parts~1 and~2. Using \eqref{56}, \eqref{60}, \eqref{64} and \eqref{65} we have for any fixed negative real number~$t$, 
\bes{
\mel\pp[(\hbe-\beta)/\sigma_* \leq t] \\
&=\pp_{\beta}[f(X/n) \leq u(\beta+t\sigma_*)]\\
& =\sum_{j \in J_*} \pp_{\beta}[f(X/n) \leq u(\beta+t\sigma_*) \given \kA_j] p_j  +  \bigo(\tau_*+\tau'_*)\\
& = \sum_{j \in  J_2^-}  \pp \left[ Y_j \leq \sum_{k \in J_2^-} \frac{f'(a_k)}{f'(a_j)}p_k^-(t)\E Y_k(t)   \right] p_j \I[f'(a_j)\neq 0] \\
& \qquad + \bigo((n\sigma_*)^{-1/2})+ \bigo(\tau_*+\tau'_*).
}
Note here that~$\tau^-_* \leq \tau'_*$. Similarly, for~$t>0$
\bes{
	\mel\pp[(\hbe-\beta)/\sigma_* > t] \\
	&= \sum_{j \in J_2^+} \pp \left[ Y_j > \sum_{k \in J_2^+} \frac{f'(a_k)}{f'(a_j)} p_k^+(t)\E Y_k(t)  \right]p_j \I[f'(a_j)\neq 0] \\
	& \qquad +  \bigo((n\sigma_*)^{-1/2}) + \bigo(\tau_*+\tau'_*),
}
where for~$k \in J_2^+$
\be{
	p_k^+(t) =\frac{q_k(t)}{\sum_{i \in J_2^+}q_i(t)}, \qquad q_i(t) =\int_{\IR} \exp(c_ix^{2m_i}+tb_ix) dx.
}
We recall that the term~$\bigo$  depends on~$t$, $\norm{f}_{\infty}$ and~$\norm{g}_{\infty}$. Hence,  for any fixed real number~$t\neq 0$, there is a positive constant~$C=C(t)$, such that for all~$n$ sufficiently large
\be{
|\pp [(\hbe-\beta)/\sigma_* \leq t]-\pp[U \leq t]|  \leq C[(n\sigma_*)^{-1/2} + \theta_- +\theta_+] =\lito(1),
}
where~$U$ has the distribution as 
\ban{
\bs{\label{66a}
	&\pp[U \leq t]\\
  &\quad= \sum_{j \in J_2^-}p_j \I[f'(a_j) \neq 0] \pp\bbbcls{ Y_j \leq  \sum_{k \in J_2^- } \frac{f'(a_k)}{f'(a_j)} p_k^-(t) \E Y_k(t) },\quad t < 0,
}\\[1ex]
\bs{\label{66b}
	&\pp[U > t] \\
  &\quad= \sum_{j \in J_2^+ }p_j \I[f'(a_j) \neq 0] \pp\bbbcls{ Y_j >  \sum_{k \in J_2^+ } \frac{f'(a_k)}{f'(a_j)} p_k^+(t)  \E Y_k(t) },\quad t>0,
}\\[1ex]
\bs{\label{66c}
	&\pp[U =0] \\
  &\quad= 1-\frac{1}{2} \sum_{j \in (J_2^+\cup J_2^-)} p_j \I[f'(a_j)\neq 0].
}
}
Note that the value~$\pp[U=0]=1-\pp[U<0]-\pp[U>0]$ is obtained as follows. Letting~$t\rightarrow 0^+$ and~$t\rightarrow 0^-$ in the formulas of~$\pp[U\leq t]$ and~$\pp[U > t]$, since~$\E Y_k(0)  =0$ and~$\pp[Y_j \leq 0]=1/2$,  we have
\bes{
\pp[U<0] = \frac{1}{2} \sum_{ j \in J_2^- } p_j \I[f'(a_j)\neq 0], \qquad \pp [U>0] = \frac{1}{2} \sum_{ j \in J_2^+ } p_j \I[f'(a_j)\neq 0].  
}
We finally conclude that 
\be{
(\hbe-\beta)/\sigma_* \stackrel{\law}{\longto} U,
}
and finish the proof of Theorem~\ref{48}.
\end{proof}
\begin{remark} We consider some special cases. If~$|J|=1$ then~$J_2^+=J_2^-=J$, and we denote by~$a_*$ the unique maximizer  and assume that~$f'(a_*) \neq 0$. In this case, the distribution of~$U$ is as follows. For all~$t \in \IR$,
	\be{
\pp[U \leq t] =\pp[Y \leq \E Y(t)],
  }
where, by denoting~$m_*$ the order of regularity of~$a_*$,
\be{
Y =Y(0), \qquad Y(t) \propto \exp(c_*x^{2m_*}+tf'(a_*)x), \quad c_* =\frac{A^{2m_*}(a_*)}{(2m_*)!} < 0.
}
Note that if~$m_*=1$ then~$Y(t)\sim N\bclr{tf'(a_*)/2|c_*|,1/2|c_*|}$, and  we can compute
\be{
U = N\left(0, 2|c_*|/f'(a_*)^2\right).
} 
Next, consider the case all the maximizers have the same order of regularity, i.e.~$m_j=m_*$ for all~$j \in J$. Then~$J_2^-=J_1^-=J_-=\{j\in J: f(a_j) =\min_{k \in J} f(a_k) \}$, and~$J_2^+=J_1^+=J_+=\{j\in J: f(a_j) =\max_{k \in J} f(a_k) \}$, and we assume that there exist~$j \in J_-$ and~$k \in J_+$ such that~$f'(a_j)f'(a_k) \neq 0$. The law of~$U$ is given as in \eqref{66a}--\eqref{66c} when replacing~$J_2^-$ and~$J_2^+$ by~$J_-$ and~$J_+$.

\noindent Finally, we consider the case~$m_j=1$ for all~$j \in J$, and
\ba{
	&c_j=c_k =c_-,  \quad f'(a_j)=f'(a_k) =d_- \quad \text{for all~$k,j \in J_{-}$,}\\
	&c_j=c_k =c_+,  \quad f'(a_j)=f'(a_k) =d_+ \quad \text{for all~$k,j \in J_{+}$.}
}
Then for~$j \in J_-$ and~$t<0$, we have~$p_j^-(t)=1/|J_-|$, and $Y_j(t) \sim N(\tfrac{td_-}{2|c_-|},\tfrac{1}{2|c_-|})$. Therefore, for~$t \in \R_-$,
\be{
	\IP[U\leq t] =p_{-}\IP\left[N \left(0,\tsfrac{1}{2|c_-|}\right) \leq \tsfrac{td_-}{2|c_-|}\right] =p_{-} \IP\left[N \left(0,\tsfrac{2|c_-|}{d_-^2}\right) \leq t\right],
}
where
\be{
	p_{-} =\sum_{j \in J_{-}}p_j \I[f'(a_j) \neq 0].
}
Similarly for~$t \in \R_+$,
\be{
	\IP[U > t] =\IP\left[N \left(0,\tsfrac{2|c_+|}{d_+^2}\right) > t\right]p_{+}, \qquad p_{+} =\sum_{j \in J_{+}}p_j \I[f'(a_j) \neq 0].
}
Thus 
\be{
	U = \tsfrac{p_{-}}{2} N^{-} \left(0,\tsfrac{2|c_-|}{d_-^2}\right) + \tsfrac{p_{+}}{2} N^{+} \left(0,\tsfrac{2|c_+|}{d_+^2}\right) + \left(1-\tsfrac{p_{-}+p_{+}}{2}\right) \delta_0,
}
where recall that $N^{-}(0,\sigma^2)$ (resp. $N^{+}(0,\sigma^2)$) is negative (resp. positive) half-normal distribution. 
\end{remark}

\section{Some examples} \label{sec:exp}

In this section, we apply Theorems~\ref{thm1}--\ref{thm3}, and~\ref{thm4} to several mean-field mixed spin models, including the homogeneous $p$-spin interaction model, and the three-spin, four-spin, and six-spin interaction models, as well as the annealed Ising model on random regular graphs. We will demonstrate that the mixed spin models exhibit a rich phase diagram for the scaling limits of magnetization. The divergence of fluctuations in these models arises from the complex structure of the maximizers of the associated function $A$. These maximizers may be unique or multiple, with the same or different orders of regularity, as described in Assumption~\textnormal{\ref{A1}}.  The annealed Ising model on regular graphs presents an interesting case where the leading term in the Hamiltonian does not take the exact form $f(\-\omega_+)$  as in mixed spin models, but instead appears in an approximate form  $f_n(\-\omega_+)$, where  $(f_n)_{n\geq 1}$  a sequence converging to a smooth function.  

Before delving into the specific models, we summarize a few points regarding presentation:

$\bullet$  We translate our results for $X_n$ to the magnetization $M_n=\omega_1 + \ldots+ \omega_n$ via the relation  $M_n=2X_n-n$.

$\bullet$ In the statements that follow, when we say that the magnetization is concentrated around points $(z_i)_{i=1}^k$, we are referring to Theorems~\ref{thm1} and~\ref{thm2}. When we say that the (conditional) central limit theorems hold, we are referring to Theorem~\ref{thm3}.

$\bullet$ A maximizer~$a_*$ of a smooth function~$A$ is said to be $2m$-regular (with~$m \in \IN$) if~$A^{(k)}(a_*) =0$ for~$k=1, \ldots, 2m-1$ and~$A^{(2m)}(a_*)<0$.

\subsection{Mean field mixed spin models}
Given~$\bob=(\beta_1,\ldots,\beta_k)\in \IR^k$ and~$\bop=(p_1,\ldots,p_k) \in \IN^k$, the mean-field mixed spin model is defined by the following Hamiltonian:
\be{
	H_n(\omega)= \sum_{j=1}^k\frac{\beta_j}{n^{p_j-1}} \sum_{1 \leq i_1, \ldots, i_{p_j} \leq n} \omega_{i_1} \ldots \omega_{i_{p_j}}  = n f_{\bop,\bob}(\-\omega_+), 
}
where
	\be{ 
		f_{\bop,\bob}(a)=\sum_{i=1}^k\beta_i(2a-1)^{p_i}.
	}
 As explained in Remark \ref{rem:a4}, this Hamiltonian satisfies the conditions \textnormal{\ref{A1}--\ref{A4}}. Hence, we can apply our theorems  to this model. The remaining  task is to analyze the maximizers of the associated function 
 $$A(a)=f_{\bop,\bob}(a)+I(a),$$ 
 and check the non-degeneracy condition of~$f_{\beta_i}(a)=(2a-1)^{p_i}$ at these points. Understanding these maximizers for general case of $\bob$ and $\bop$ is highly non-trivial and  warrants independent research. We aim to investigate some particular cases. First, we consider the case where  only the $p$ spin interactions are allowed. This model was proposed and studied in \cite{MSB1, MSB2}. Second, we analyze the cubic model where  two-spin and three-spin interactions are mixed. This model has been investigated in \cite{CMO}. Third, we propose an interesting four-spin interaction model where the phase diagram for scaling of magnetization becomes complex, ranging from $n^{1/2}$, $n^{1/4}$ to $n^{1/6}$.   Finally, we offer an example  where the associated function $A$ has two maximizers with different  orders of regularity. Consequently, the conditional limit theorems at the maximizers occur at different scales.

\subsubsection{$p$-spin Curie-Weiss model} We consider the homogeneous $p$-spin interaction with Hamiltonian 
 
\ben{ \label{67}
H_n(\omega)=  n f_{\beta,h}(\-\omega_+), \qquad 	f_{\beta,h}(a)=\beta(2a-1)^p +h(2a-1),
}
where $\beta >0$ and $h \in \R$ are parameters.  
\cite{MSB1} have fully characterized the maximizers of the function~$A=f_{\beta,h}+I$ by showing that the parameter space~$(\beta,h) \in \R_+ \times  \R$ is partitioned into disjoint regions:
\begin{itemize}
	\item [$\bullet$] regular region~$R_1=\{(\beta,h): A \textrm{ has an unique maximizer } a_* \in (0,1) \}$ (in this case~$a_*$ is~$2$-regular);
	\item[$\bullet$]~$p$-critical curve~$R_2=\{(\beta,h): A \textrm{ has multiple maximizers  } 0<a_1< \ldots<a_k<1\}$ (in this case all the maximizers are~$2$-regular);
	\item[$\bullet$]~$p$-special points~$R_3=\{(\beta,h): A \textrm{ has an unique maximizer }a_* \in (0,1)$, $A''(a_*)=0 \}$ (in this case~$a_*$ is~$4$-regular).
\end{itemize} 
We refer the reader to Appendix B of \cite{MSB1} for a complete picture of the partition~$(R_1,R_2,R_3)$.

Now, given the additional parameters~$(\bb,\bh)$, \cite{MSB2} considered the perturbed Hamiltonians 
\ba{
	H_n^r(\omega) &= n f_{\beta,h} (\-\omega_+) +\sqrt{n} B(\-\omega_+)  \\
	H_n^s(\omega) &= n f_{\beta,h}(\-\omega_+)+ n^{1/4}B(\-\omega_+), 
}
where 
\be{
B(a)=f_{\bb,\bh}(a), \quad  a  \in [0,1],
}
with~$f_{\bb,\bh}$ defined as in \eqref{67}.  Denoting the corresponding Gibbs measures by~$\mu_n^r$ and~$\mu_n^s$ and using Theorems~\ref{thm1} and~\ref{thm3},    we obtain the following result.

\begin{theorem} Consider the magnetization under the perturbed measures~$\mu_n^r$ and~$\mu^s_n$. 
	\begin{itemize}
		\item [(i)] If~$(\beta,h) \in R_1$ then under $\mu_n^r$,     
		the magnetization is concentrated around $2a_*-1$. Moreover, the  central limit theorem  holds with Weierstrass distance $\bigo(1/\sqrt{n})$.
		\item [(ii)] If~$(\beta,h) \in R_2$ then under $\mu_n^r$, the magnetization is concentrated around the  points $(2a_i-1)_{i=1}^k$. Moreover, the conditional central limit theorems around these points hold with Weierstrass distance $\bigo(1/\sqrt{n})$.
		\item [(iii)] If~$(\beta,h)\in R_3$ then under~$\mu^s_n$,
		\be{
			\dw\clr{W_n, Y}= \bigo(n^{-1/4}), \qquad 	W_n= \frac{M_n-n(2a_*-1)}{n^{3/4}}, 
		}
		where
		\be{
			Y  \propto \exp\Big(\frac{c_*x^4}{16} +\frac{b_* x}{2} \Big), \quad c_*=\frac{A^{(4)}(a_*)}{24}, \quad b_*=B'(a_*).
		}
	\end{itemize}
\end{theorem}
The above theorem covers  Theorem 2.1 of \cite{MSB1} (the main result in this paper) and  Theorem 3.1 of \cite{MSB2} (the key result leading to the maximum likelihood estimators). 

Now we aim to  apply Theorem~\ref{thm4} to find the  scaling limits of MLEs. First, we have to check the non-degeneracy condition in \eqref{49}.  Observe that this condition is always true for the parameter~$h$, since the corresponding  function~$f_h(a)=2a-1$ is not degenerated at any~$a\in [0,1]$. However, that  condition for~$\beta$ does not hold when~$\beta \leq \tilde{\beta}_p$ and~$h=0$, where~$\tilde{\beta}_p=\sup\{ \beta \geq 0: \sup_{a \in [0,1]} A(a) =0 \}$. In fact, in this case~$a=1/2$ is a maximizer of~$A$ that belongs to the set~$J_-$, and the corresponding function~$f_{\beta}(a)= (2a-1)^p$ is degenerated at this point.  In summary, we have the following.

\begin{theorem} Consider the maximum likelihood estimators of the~$p$-spin Curie-Weiss model denoted by~$\hbe_n$ and~$\hat{h}_n$.
	\begin{itemize}
		\item [(ia)] If~$(\beta, h) \in R_1$, then 
		\be{			\sqrt{n}(\hat{h}_n-h) \overset{\law}{\longto} N(0,\sigma_{h}),
		}
		with~$\sigma_{h}$ a positive constant.
		\item[(ib)] If~$(\beta, h) \in R_1 \setminus \{(\beta,0): \beta \leq   \tilde{\beta}_p\}$, then
		\be{
			\sqrt{n}(\hbe_n-\beta) \overset{\law}{\longto} N(0,\sigma_{\beta}), 
		}
		with~$\sigma_{\beta}$ a positive constant.
		\item [(iia)] If~$(\beta, h) \in R_2$, then 
		\be{
			\sqrt{n}(\hat{h}_n-h) \overset{\law}{\longto} U_h, 
		}
		where
		\be{
			U_h= p^-_h  N^-(0,\sigma_h^-) + p_h^+ N^+(0,\sigma_h^+) + (1-p^-_h-p_h^+) \delta_0,	
		}
		with~$p^{\pm}_h, \sigma^{\pm}_h$   positive constants. 
		\item [(iib)] If~$(\beta, h) \in R_2 \setminus \{(\tilde{\beta}_p,0)\}$, then 
		\be{
			\sqrt{n}(\hat{\beta}_n-\beta) \overset{\law}{\longto} U_\beta, 
		}
		where
		\be{
			U_{\beta}= p^-_\beta  N^-(0,\sigma_\beta^-) + p_\beta^+ N^+(0,\sigma_\beta^+) + (1-p^-_\beta-p_\beta^+) \delta_0,	
		}
		with~$p^{\pm}_\beta, \sigma^{\pm}_\beta$   positive constants. 
   	\item [(iii)] If~$(\beta,h)\in R_3$, then 
		\be{
			n^{3/4}(\hbe_n-\beta) \overset{\law}{\longto}Z_{\beta}, \qquad 
			n^{3/4}(\hat{h}_n-h) \overset{\law}{\longto}Z_h,
		}
		where for~$\nu \in \{ \beta, h\}$ the random variable~$Z_{\nu}$ has the distribution  
		\be{
			\IP[Z_{\nu} \leq t] = \pp[Y_{\gamma}(0) \leq \E Y_{\nu}(t)],
		}
		where 
		\be{
		 Y_{\nu}(t) \propto  \exp\bclr{c_*y^{4}+tf_{\nu}'(a_*)y},
		}
		with~$c_*= A^{(4)}(a_*)/24$, and~$f_{\beta}(a)=(2a-1)^p$ and~$f_{h}(a)=(2a-1)$.
	\end{itemize}
\end{theorem}
Note that in (iib), all the points $(\beta, 0)$ with $\beta < \tilde{\beta}_p$ are not in $R_2$ (in fact, these points are in $R_1$).  The above result covers Theorems 2.2--2.7 of \cite{MSB2}, except for the estimator~$\hbe_n$ when~$h=0$ and~$\beta \leq \tilde{\beta}_p$, which is corresponding to the results (2.19), (2.22) and (2.26) in this paper.

\subsubsection{Cubic mean field Ising model}
 \cite{CMO} consider a model combining three-spin  and two-spin  interactions as follows:
\ben{ \label{hcub}
	f_{\beta,h}(a)=\beta(2a-1)^3 +h(2a-1)^2, 
}
where $\beta >0$ and $h \in \R$ are parameters. The complete phase diagram of the model (or equivalently the maximizers of the associated function $A=f_{\beta,h}+I$)  has been shown in \cite[Proposition 2.2]{CMO}. More precisely, there exists a  curve parameterized by a function $g$, say $\gamma= \{(\beta,h): h=g(\beta), \beta >0\}$   such that 
\begin{itemize}
	\item [$\bullet$] if $(\beta, h) \in \R_+ \times \R \setminus \gamma $ then $A$ has an unique maximizer  $a_* \in (0,1)$ satisfying $a_* \neq 1/2$  and $a_*$ is~$2$-regular;
	\item[$\bullet$] if $(\beta, h) \in \gamma$ then $A$ has two  maximizers  $1/2=a_-<a_+<1 $ which are all $2$-regular.
\end{itemize} 
Notice  that   $g(\beta) \rightarrow 1/2$ as $\beta \rightarrow 0$, and the model turns to be  the standard critical Curie-Weiss model. In this case,   $A$ has the unique   maximizer  $a_*=1/2 $ which is $4$-regular and the fluctuation of model has been well known.

\begin{theorem} Consider the magnetization of the cubic mean field model. 
	\begin{itemize}
		\item [(i)] If~$(\beta,h) \in \R_+ \times \R \setminus \gamma$  the magnetization is concentrated around $2a_*-1$. Moreover, the  central limit theorem  holds with Weierstrass distance $\bigo(1/\sqrt{n})$. 
		\item [(ii)] If~$(\beta, h) \in \gamma$ then
		the magnetization is concentrated around two points $2a_-1$ and $2a_+-1$. Moreover, the conditional central limit theorems around these points hold with Weierstrass distance $\bigo(1/\sqrt{n})$. 
	\end{itemize}
\end{theorem}

Next,  we consider the fluctuation of MLEs. Since the functions $f_{\beta}(a)=(2a-1)^3$ and $f_h(a)=(2a-1)^2$ are degenerated at  $a=1/2$, Theorem \ref{thm4} is not applicable for the case $(\beta,h) \in \gamma$, where $1/2$ is a maximizer. For the remaining case, we have the Gaussian fluctuation as follows.

\begin{theorem} Consider the maximum likelihood estimators of the cubic mean field  model denoted by~$\hbe_n$ and~$\hat{h}_n$. If $(\beta,h)\in  \R_+ \times \R \setminus \gamma $ then 
\be{
\sqrt{n}(\hat{\beta}_n-h) \overset{\law}{\longto} N(0,\sigma_{\beta}), \qquad 
\sqrt{n}(\hat{h}_n-h) \overset{\law}{\longto} N(0,\sigma_{h}),
		}
		where~ $\sigma_\beta$ and $\sigma_{h}$  are positive constants.	
\end{theorem}

\subsubsection{Four-spin interaction mean field model}
We consider a model that incorporates a mixture of four-spin and two-spin interactions as follows
\ben{ \label{hcub}
	f_{\beta,h}(a)=\beta(2a-1)^4 +h(2a-1)^2, 
}
where $\beta >0$ and $h \in \R$ are parameters. The phase diagram of the model (or of the associated function $A(a)=f_{\beta,h}(a)+I(a)$) is given by the following proposition. See  Figure \ref{phase4} for an illustration. 
\begin{proposition} \label{prop:fours}
There exists a curve  parameterized by a function $g$, say $\gamma= \{(\beta,h): h=g(\beta), \beta >0\}$ satisfying $g(\beta)=1/2$ for all $\beta \leq 1/12$ and  
\begin{itemize}
       \item[$(i)$] if $(\beta, h) \in R_1:=\{(\beta,h): h<g(\beta) \}$ then $A$ has the unique   maximizer $1/2$, which is $2$-regular;
	\item [$(ii)$] if $(\beta, h) \in R_2:=\{(\beta,h): h>g(\beta) \}$ then $A$ has  two  maximizers  $a_\pm$ satisfying $0<a_-<1/2<a_+<1$, and both are  $2$-regular;  
     \item[$(iii)$] if $(\beta, h) \in R_3: =\{(\beta,h): h=g(\beta), \beta >1/12 \} $ then $A$ has three   maximizers $1/2, a_\pm$  satisfying $0<a_-<1/2<a_+<1$, and  all are $2$-regular;
     \item [$(iv)$] if $(\beta, h) \in R_4:= \{(\beta,h): 0<\beta < 1/12, h =1/2\} $ then $A$ has the unique maximizer  $0$, which is $4$-regular;
     \item[$(v)$] if $(\beta, h) =(1/12,1/2)$ then $A$ the unique maximizer $0$ which is $6$-regular.
\end{itemize}  
\end{proposition}
The proof of Proposition \ref{prop:fours} is given in Appendix.

\begin{figure}[h]  
\begin{center}
\includegraphics[scale=0.32]{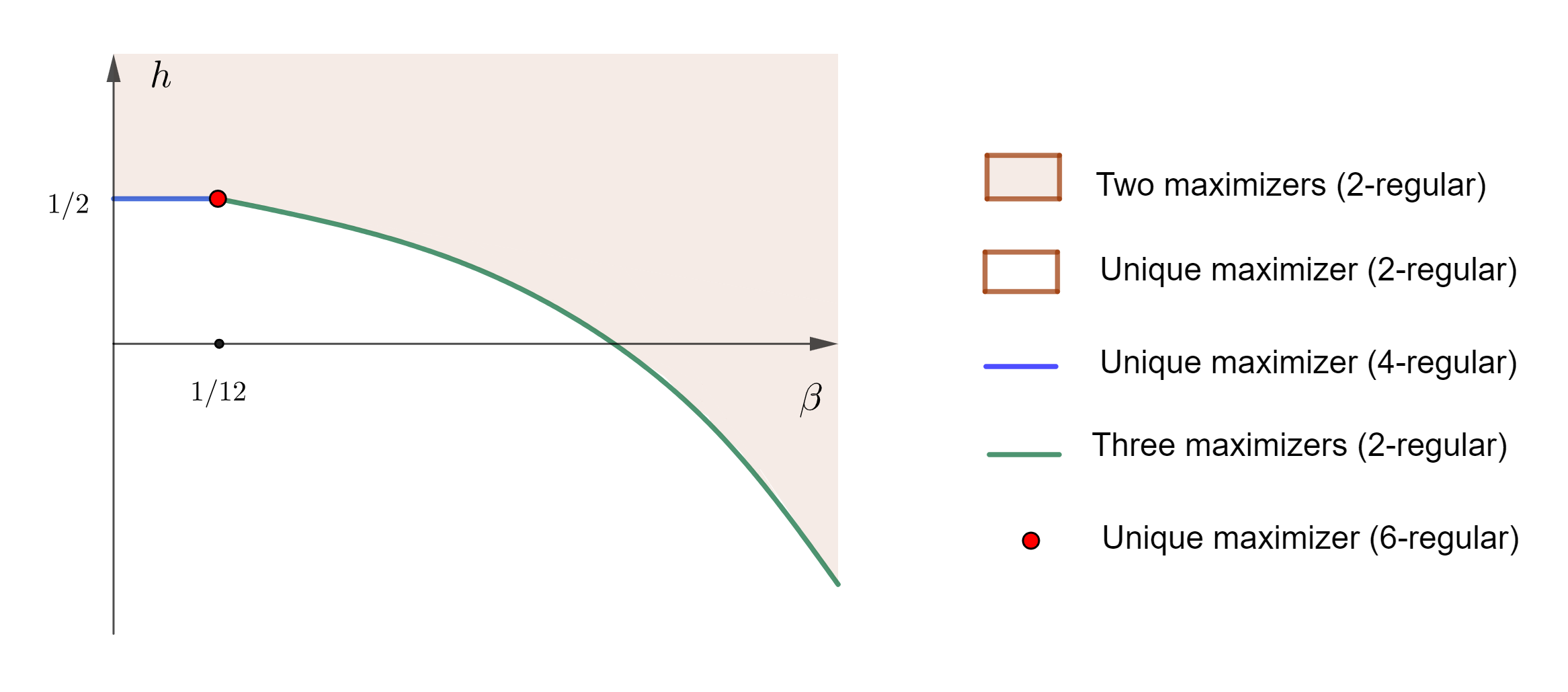}
\end{center}
\caption{Phase diagram of four-spin model} 
\label{phase4}
\end{figure}

\begin{theorem} Consider the magnetization of the four spin interaction mean field model. 
\begin{itemize}
       \item [$(i)$] If $(\beta, h) \in R_1$, the magnetization is concentrated around $0$. Moreover, the  central limit theorem  holds with Weierstrass distance $\bigo(1/\sqrt{n})$. 
	\item [$(ii)$] If $(\beta, h) \in R_2$, the magnetization is concentrated around two points $2a_--1$ and $2a_+-1$. Moreover, the conditional central limit theorems around these points hold with Weierstrass distance $\bigo(1/\sqrt{n})$.   
	\item[$(iii)$] If $(\beta, h) \in R_3$,  the magnetization is concentrated around three points $2a_-1$, $0$ and $2a_+-1$. Moreover, the conditional central limit theorems around these points hold with Weierstrass distance $\bigo(1/\sqrt{n})$.   
     \item[$(iv)$] If $(\beta, h) \in R_4$ then
     \be{
			\dw\clr{n^{-3/4}M_n, Y}= \bigo(n^{-1/4}), \qquad Y  \propto \exp\Big(\frac{A^{(4)}(0)x^4}{2^44!} \Big).
		}		
     \item [$(v)$] If $(\beta, h) =(1/12,1/2) $  then 
     \be{
			\dw\clr{n^{-5/6}M_n, Y}= \bigo(n^{-1/4}), \qquad Y  \propto \exp\Big(\frac{A^{(6)}(0)x^6}{2^66!} \Big).
		}	
\end{itemize} 
\end{theorem}

Next,  we consider the fluctuation of MLEs. Since the functions $f_{\beta}(a)=(2a-1)^4$ and $f_h(a)=(2a-1)^2$ are degenerated at  $a=1/2$, Theorem \ref{thm4} is   applicable for the case (ii) where $1/2$ is not a maximizer.

\begin{theorem} Consider the maximum likelihood estimators of the cubic mean field  model denoted by~$\hbe_n$ and~$\hat{h}_n$. If $(\beta,h)\in  R_2$ then

		\be{
		\sqrt{n}(\hat{\beta}_n-\beta) \overset{\law}{\longto} U_\beta
  \qquad
  \sqrt{n}(\hat{h}_n-h) \overset{\law}{\longto} U_h, 
		}
		where
  \be{
			U_\beta= p^-_\beta  N^-(0,\sigma_\beta^-) + p_\beta^+ N^+(0,\sigma_\beta^+) + (1-p^-_\beta-p_\beta^+) \delta_0,	
		}
		\be{
			U_h= p^-_h  N^-(0,\sigma_h^-) + p_h^+ N^+(0,\sigma_h^+) + (1-p^-_h-p_h^+) \delta_0,	
		}
		with~$p^{\pm}_h, \sigma^{\pm}_h$   $p^{\pm}_\beta, \sigma^{\pm}_\beta$ positive constants.
\end{theorem}

\subsubsection{A six-spin interaction model with varying regularity}
We aim to construct an example of a mixed interaction model, where the magnetization concentrates at two distinct points, each with a different degree of concentration. Specifically, we consider the following six-spin interaction model:
\ben{ \label{sixth}
f_{\beta,h}(a) = \beta a^6+ ha^5+a^2/2. 
}
\begin{theorem} \label{theo:sixth}
Consider the model of six-spin interaction \eqref{sixth} with  $\beta$ and $h$ given in \eqref{dob} and \eqref{doh}. Then the associated function $A=f_{\beta,h}+I$ has two maximizers $1/2$ and $0.95$, where $1/2$ is $4$-regular and $0.95$ is $2$-regular.  As a consequence,
\be{
			M_n/n    \stackrel{\law}{\longto} 0,	\quad M_n=2X_n-n.
		}
		 Moreover, 
		\be{
			\dw \left( \law\clr{n^{-3/4}M_n \given \-\omega_+ \in (\tfrac{1}{2}-\delta, \tfrac{1}{2}+\delta ) },Y \right)= \bigo(n^{-1/4}), \quad Y \propto   	\exp\Big(\frac{A^{(4)}(\tfrac{1}{2})x^4}{2^44!} \Big),
		}
and 

\be{
			\dw \left( \law\clr{W_n \given \-\omega_+ \in (0.95-\delta, 0.95+\delta ) },N \left(0,\tsfrac{4}{|A''(0.95)|} \right) \right)= \bigo(n^{-1/2}),
		}
  where $\delta$ is a small constant and
		\be{
			W_{n}= \frac{M_n-0.9 n}{\sqrt{n}}.
		}
\end{theorem}
\begin{figure} 
\begin{center}
\includegraphics[scale=0.7]{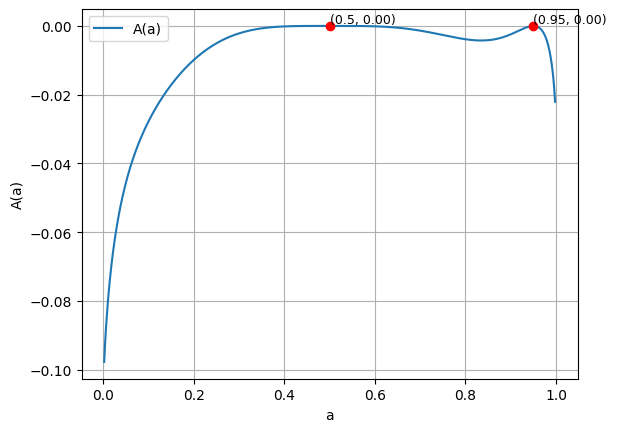}
\end{center}
\caption{Plot of the function $A$ in a six-spin model}
\label{sixspin}
\end{figure}
The detailed analysis of maximizers of $A$  is put in Appendix and the plot of $A$ is given in Figure \ref{sixspin}.

\subsection{Annealed Ising model on random  regular graphs}
Let~$G_n=(V_n, E_n)$ be the random regular graph of degree~$d \geq 3$ with~$n$ vertices~$V_n=\{v_1,\ldots,v_n\}$. The  Gibbs measure of annealed Ising model is defined as follows. For~$\omega \in \{1,-1\}^n$,
$$\mu_n(\omega) \propto \IE\clc{\exp(H_n(\omega))}, \qquad H_n(\omega) = \beta \sum_{(v_i,v_j) \in E_n} \omega_i \omega_j + h \sum_{i=1}^n \omega_i,$$
where expectation is taken over the space of random regular graphs with respect to a uniform distribution. \cite[Eq. (3.2) and Lemma 2.1]{C1} proved that if~$\-\omega_+=k/n$ then
\be{
\mu_n(\omega) \propto \exp(2hk)g(\beta, dk, dn),
}
where~$\{g(\beta, m, l)\}_{m \leq l}$ satisfies that 
\bgn{ \label{68}
	\babs{l^{-1} \log g(\beta, m, l)  -g_{\beta}(m/l)} =\bigo(1/l),\\
	\babs{\bclr{l^{-1}\log g(\beta, m, l)  -g_{\beta}(m/l)}-\bclr{l^{-1} \log g(\beta, k, l)  -g_{\beta}(k/l) }} = \bigo\bclr{|k-m|/l^2},  \notag 
}
with 
\be{
	g_{\beta}(a)=\int \limits_0^{a\wedge(1-a)} \frac{e^{-2\beta}(1-2s)+\sqrt{1+(e^{-4\beta}-1)(1-2s)^2}}{2(1-s)} ds.
}
Therefore, with~$X_n=n \-\omega_+$, we have
\be{
	\mu_n(X_n=k) \propto \exp(nA_n(k/n))
}
with 
\be{
	A_n\big(k/n\big) = 2h k/n + \frac{1}{n} \log  g(\beta, dk, dn) + \frac{1}{n} \log \binom{n}{k}.
}
By \eqref{68} the function~$A_n$ is well approximated by~$A: [0, 1] \rightarrow \R$ given as
\be{
	A(a) = 2h a+  dg_{\beta}(a) +I(a).
}
In particular,  we can find positive constants~$\eps_*$, $\delta_*$ and~$C_*$ such that the conditions \textnormal{\ref{A1}--\ref{A3}} hold. \cite[Claim 1*]{C1} and \cite[Lemma 2.2]{C2} showed that 
\begin{itemize}
	\item [$\bullet$] if~$(\beta,h) \in \kU=\{(\beta, h): \beta >0, h\neq 0, \textrm{or } 0< \beta <\beta_c, h=0\}$ then~$A$ has a unique $2$-regular maximizer~$a_*\in (0,1)$;
	\item[$\bullet$] if~$\beta >\beta_c$ and~$h=0$ then~$A$ has two $2$-regular maximizers~$0<a_-<a_+=1-a_-<1$;
	\item[$\bullet$] if~$\beta =\beta_c$ and~$h=0$ then~$A$ has the unique $4$-regular maximizer~$a_*=1/2$.
\end{itemize}
Here~$\beta_c$ is the critical value of the model~$\beta_c=\mathop{\mathrm{atanh}}(1/(d-1))$. We now verify~\ref{A4} for the case~$(ii)$.  Since~$h=0$, the model is symmetric and thus~$\mu_n(\omega) =\mu_n(-\omega)$ and 
\ben{ \label{69}
	\mu_n(X_n=k) = \mu_n(X_n=n-k).
}
Letting~$k_-=[na_-]$ and~$k_+=[na_+]$, we aim to  show 
\ben{ \label{70}
	|A_n(k_-/n)-A_n(k_+/n)| =\bigo(n^{-3/2}).
}
Indeed, using  \eqref{69} and~\ref{A3}
\bes{
	\mel\bigr|A_n(k_-/n)-A_n(k_+/n)\bigl| \\
  & =\bigr|A_n((n-k_-)/n)-A_n(k_+/n)\bigl|\\
	&=\bigr|A((n-k_-)/n)-A(k_+/n)\bigl|+\bigo(|n-k_--k_+|/n^2)\\
	&=\bigo\bclr{((n-k_-)/n-a_+)^2}+ \bigo\bclr{(k_+/n-a_+)^2} + \bigo\bclr{|n-k_--k_+|/n^2}\\
	&=\bigo\clr{n^{-2}}.
}
Here, for the third line, we used Taylor expansion at~$a_+$ and~$A'(a_+)=0$, and for the last one, we used~$k_{\pm}=[na_{\pm}]$ and~$a_-+a_+=1$. Therefore, \eqref{70} holds when~$h=0$ and~$\beta> \beta_c$. 

In conclusion, all the conditions \textnormal{\ref{A1}--\ref{A4}} hold, and thus using Theorems~\ref{thm1}  and~$M_n=2X_n-n$, we have the following. 
\begin{theorem} Consider the annealed Ising model on a random regular graph.
	\begin{itemize}
		\item [(i)] If~$(\beta,h) \in \kU$, the magnetization is concentrated around $2a_*-1$. Moreover, the  central limit theorem  holds with Weierstrass distance $\bigo(1/\sqrt{n})$. 
		\item [(ii)] If~$\beta> \beta_c$ and $h=0$ then 
		the magnetization is concentrated around two points $2a_--1$ and $2a_+-1$. Moreover, the conditional central limit theorems around these points hold with Weierstrass distance $\bigo(1/\sqrt{n})$.
		\item [(iii)] If~$\beta= \beta_c$ and $h=0$ then 
		\be{
			\dw\bclr{n^{-3/4}M_n, Y}=\bigo(n^{-1/4}), \qquad 	Y \propto   	\exp\Big(\frac{A^{(4)}(0)x^4}{2^44!} \Big). 
		}
		\end{itemize}	
\end{theorem}
Parts $(i)$ and $(ii)$ are the main results of \cite[Theorem 1.3]{C1} and Part~$(iii)$ is the main result of \cite[Theorem 1.3]{C2} with a convergence rate. The model is not linear in~$\beta$ but linear in~$h$, and hence we can also prove the following.
\begin{theorem} Consider the maximum likelihood estimator~$\hat{h}_n$ of the annealed Ising model on random regular graphs. 
\begin{itemize}
	\item [\rm{(i)}] If~$(\beta, h) \in \kU$ then 
	\be{
		\sqrt{n}(\hat{h}_n-h) \overset{\law}{\longto} N(0,\sigma_{h}),
	}
	with~$\sigma_{h}$ a positive constant.
	\item [\rm{(ii)}] If~$\beta > \beta_c$ and $h=0$ then 
  \be{
	\sqrt{n}(\hat{h}_n-h) \overset{\law}{\longto} U_h, 
}
where
\be{
	U_h= p^-_h  N^-(0,\sigma_h^-) + p_h^+ N^+(0,\sigma_h^+) + (1-p^-_h-p_h^+) \delta_0,	
}
with~$p^{\pm}_h, \sigma^{\pm}_h$   positive constants. 
	\item [\rm{(iii)}] If~$\beta= \beta_c$, $h =0$ then 
	\be{
		n^{3/4}(\hat{h}_n-h) \overset{\law}{\longto}Z_h, 
	}
	where~$Z_h$ has the distribution  as
	\be{
		\IP[Z_h \leq t] =\pp[Y_h(0) \leq \E[Y_h(t)]]
	}
	with~$Y_{h}(t) \propto  \exp(c_*y^{4}+2ty)$ and~$c_*=A^{(4)}(1/2)/24$. 
\end{itemize}
\end{theorem}

\section*{Appendix}
\subsection*{Proof of Proposition \ref{prop:fours}}
We analyze the maximizers of $A(a)=\beta(2a-1)^4+h(2a-1)^2+I(a)$ for $a \in [0,1]$. For the convenience, we change the variable $t=2a-1$ and the function $A$ turns to be $F: [-1,1] \rightarrow \R$ given by
\[
F(t)= \beta t^4 + ht^2 + E(t), \qquad E(t)= -\frac{1+t}{2} \log (1+t) - \frac{1-t}{2} \log (1-t).
\]
We observe that
\ben{ \label{fpo}
F(t)= t^2(h-g_\beta(t)), \qquad  F'(t)=2t(h-\hat{g}_\beta(t)),
}
where 
\ben{ \label{gbt}
 g_\beta(t) =-\frac{E(t)}{t^2}- \beta t^2, \qquad 
\hat{g}_\beta(t) =\frac{\atanh(t)}{2t}-2 \beta t^2.
}
We define for $\beta >0$: 
\[
g(\beta) = \inf_{s \in [-1,1]} g_\beta(s), \qquad \hat{g}(\beta) = \inf_{s \in [-1,1]} \hat{g}_\beta(s). 
\]
The properties of  $g_\beta$, $\hat{g}_\beta$, $g$ and $\hat{g}$ are summarized as follows.
\begin{lemma} \label{l:gg}
The following assertions hold. 
\begin{itemize}
\item[(i)] If $\beta \leq 1/12$, then $\hat{g}_\beta$  has the unique critical point $0$, which is the minimizer. If $\beta >1/12$, then $g_\beta$ has  a local maximizer $0$, and two  symmetry minimizers denoted by $s_\pm$ satisfying  $q(s_\pm)=\beta$ with 
\ben{ \label{eqq}
q(s)= \frac{1}{8s^2(1-s^2)} - \frac{\atanh(s)}{8s^3}.
}
\item[(ii)] Both functions $g$ and $\hat{g}$ are non-increasing and satisfy $g(\beta)=\hat{g}(\beta)=1/2$ for $\beta \leq 1/12$, and $\hat{g}(\beta) \leq g(\beta)$ for all $\beta >0$.
\item [(iii)] If $h<g(\beta)$ then $0$ is the unique maximizer of $F$. If $\beta >1/12$ and $h> g(\beta)$ then $0$ is not a maximizer of $F$. 
\end{itemize}
\end{lemma}
Assuming the above lemma, we complete the proof of Proposition \ref{prop:fours}.\\ 

\noindent \textbf{Case 1}: $h< g(\beta)$. By Lemma \ref{l:gg}  (iii), $0$ is the unique maximizer of $F$. Moreover, $F''(0)=2h-1<2g(\beta)-1 \leq 0$, using Lemma \ref{l:gg} (ii).  Hence, the maximizer $t=0$ is $2$-regular. \\ 

  \begin{figure}
	\centering
	\subfloat[$\beta >1/12$]{\includegraphics[scale=0.35]{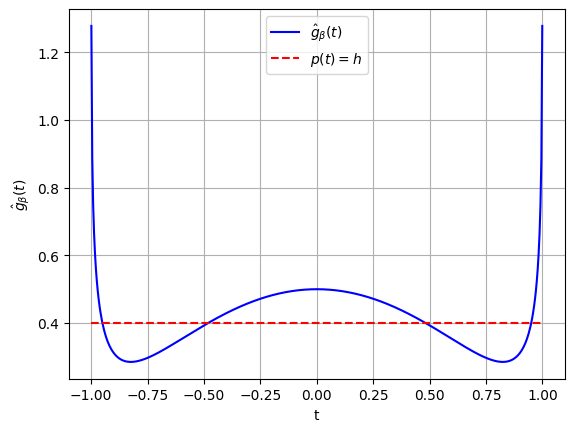}}\hspace{0.5 cm}
	\subfloat[$\beta\leq 1/12$]{\includegraphics[scale=0.35]{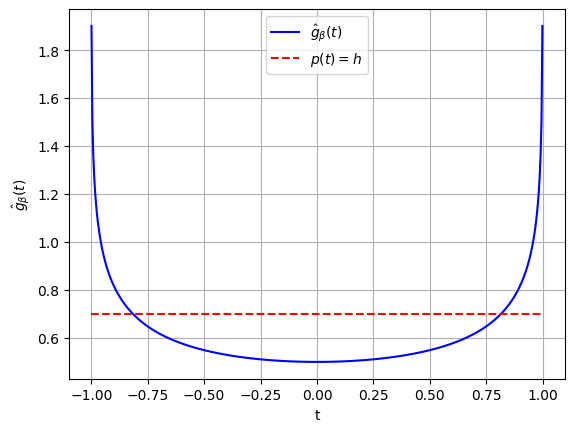}} 
 \caption{Plot of the equation $h=\hat{g}_\beta(t)$}
 \label{fgbt}
	\end{figure}

\noindent \textbf{Case 2}: $h> g(\beta)$. We consider three subcases.

\underline{Case 2a}: $h> 1/2$. The equation  $h=\hat{g}_\beta(t)$ has two symmetry solutions $t_\pm \in (-1,1) \setminus \{0\}$, see Figure \ref{fgbt}  for an illustration.  Hence, $F'(t)=0$ has three solutions $t_\pm$ and $0$. Moreover, since $g_\beta(1) = + \infty$ we have $h-g_\beta(t)<0$ and hence $F'(t)<0$ for all $t>t_+'$. Similarly, $F'(t)>0$ for all $t<t_-'$. Thus $t_\pm$ are local maximizers and $0$ is a local minimizer. Additionally, $F$ is symmetry, and so    $t_\pm$ are actually the  maximizers.  Since $1/2<h=\hat{g}_\beta(t_+)=\atanh(t_+)/(2t_+)-2\beta t_+^2$, we have $\beta t_+^2 \leq \atanh(t_+)/(4t_+)-1/4$. Thus,
   \bea{
     F''(t_+) &=& 12 \beta t_+^2 +2h - \frac{1}{1-t_+^2} = \frac{\atanh(t_+)}{t_+} + 8\beta t_+^2 - \frac{1}{1-t_+^2}\\
     &\leq& 3\frac{\atanh(t_+)}{t_+} -2  - \frac{1}{1-t_+^2}.
     }
      Consider $z(t)=3\atanh(t)-2t -t/(1-t^2)$ for $t \in (0,1)$. We have $z'(t)=-2t^4/(1-t^2)^2<0$. Hence, $z(t_+) < z(0)=0$, so $F''(t_+)=z(t_+)/t_+<0$. We then  have  $F''(t_-)=F''(t_+)<0$. In summary, the maximizers  $t_\pm$ are  $2$-regular.

 \underline{Case 2b}: $1/2=h>g(\beta)$. By Lemma \ref{l:gg} (ii),  $\beta >1/12$, since otherwise $g(\beta)=1/2$. Then the equation $h=\hat{g}_\beta(t)$ has three solutions $t_\pm$ and $0$. Using the same arguments as in Case 2a, we  obtain that $0$ is a local minimizer and  $t_\pm$ are maximizers which are $2$-regular.

\underline{Case 2c}: $1/2>h>g(\beta)$. We also have $\beta >1/12$ and the equation $h=\hat{g}_\beta(t)$ has four solutions $t_\pm$, $t_\pm'$ satisfying $-1<t_-<s_-<t'_-<0<t'_+<s_+<t_+<1$ and $t_\pm=-t_\pm'$, where recall  that $s_\pm$ are minimizers of $\hat{g}_\beta$ and are solutions of equation $q(s)=\beta$ given in \eqref{eqq}, see Figure \ref{fgbt}  for an illustration.  Using similar as for Case 2a, we conclude that  $t_-, t_+$ and $0$ are local  maximizers of $F$. By Lemma \ref{l:gg} (iii), $0$ is not a maximizer. Hence, $t_\pm$ are maximizers of $F$.  Since $h=\hat{g}_\beta(t_+)=\atanh(t_+)/(2t_+)-2\beta t_+^2> \hat{g}_\beta(s_+)$, we have $\beta t_+^2 \leq \atanh(t_+)/(4t_+)-\hat{g}_\beta(s_+)/2$. Thus
 \bea{
     F''(t_+) &=& 12 \beta t_+^2 +2h - \frac{1}{1-t_+^2} = \frac{\atanh(t_+)}{t_+} + 8\beta t_+^2 - \frac{1}{1-t_+^2}\\
     &\leq& 3\frac{\atanh(t_+)}{t_+} - 4 \hat{g}_\beta(s_+)  - \frac{1}{1-t_+^2}.
     }
Set $l(t)= 3\atanh(t)/t-1/(1-t^2)$. Using $\atanh(t) \geq t+t^3/3$, we have 
\bea{
l'(t) &=& \frac{-5 t^3 - 3 (t^2 - 1)^2 \atanh(t) + 3 t}{t^2 (1 - t^2)^2} \\
&\leq& \frac{-5 t^3 - 3 (t^2 - 1)^2 (t+t^3/3) + 3 t}{t^2 (1 - t^2)^2} = -\frac{t^7+t^5}{t^2 (1 - t^2)^2}<0.
}
Therefore, $l(t_+) < l(s_+)$, since $0<s_+<t_+$. Hence,
\bea{
F''(t_+) < l(s_+)-4\hat{g}_\beta(s_+)&=& \frac{\atanh(s_+)}{s_+}-\frac{1}{(1-s_+^2)} + 8 \beta s_+^2\\
&=&8s_+^2(\beta-q(s_+))=0,
}
where we used $q(s_+)=\beta$ and recall the formula of $q$ in \eqref{eqq}. We then have   $F''(t_-)=-F''(t_+)<0$. Therefore, the mazximizers $t_\pm$ are $2$-regular. \\

\noindent \textbf{Case 3}: $h=g(\beta)$. We consider theree sub-cases.

\underline{Case 3a}: $\beta < 1/12$.  Then  $h=g(\beta)=1/2$ and  $0$ is the unique maximizer. We can check that  
\[
f'(0)=f''(0)= f'''(0)=0, f^{(4)}(0)=24 \beta -2<0.
\]
Hence, the maximizer $0$ is  $4$-regular.

\underline{Case 3b}: $\beta =1/12$. Then $h=g(\beta)=1/2$ and $0$ is the unique maximizer. We have 
\[
f'(0)=f''(0)= f'''(0)=f^{(4)}(0)=f^{(5)}(0)=0, f^{(6)}(0)=-24 <0.
\]
Hence, the maximizer $0$ is  $6$-regular.

\underline{Case 3c}: $\beta >1/12$.  For all $\eps>0$ small enough,  $h^\eps=g(\beta)+\eps \in (g(\beta),1/2)$. We denote by  $F^\eps$ for the function $F$ when replacing $h$ by $h^\eps$.  As explained in Case 2,  $F^\eps$ has two symmetry maximizers $t^\eps_\pm$ such that $0<t_-^\eps<s_-<0<s_+<t_+^\eps$ and $F^\eps(t^\eps_\pm)>F^\eps(0)=0$. By the compactness, we can take  sequence $\eps_i \rightarrow 0$ such that $(t_\pm^{\eps_i})_{i\geq 1}$ converge to symmetry limit $t_\pm$ satisfying $0<t_-\leq s_-<0<s_+\leq t_+$.  We have 
\be{
F(t_+)=\lim_{i \rightarrow \infty} F^{\eps_i}(t^{\eps_i}_+) \geq 0, \qquad F(t_-)=\lim_{i \rightarrow \infty} F^{\eps_i}(t^{\eps_i}_-) \geq 0 
}
On the other hand, since $h=g(\beta) \leq g_\beta(s)$ for all $s \in [-1,1]$, one has
\[
F(s)=s^2(h-g_\beta(s)) \leq 0 \qquad \forall \, s \in [-1,1]
\]
Thus $t_-,0, t_+$ are maximizers of $F$. Moreover, the function $F$ has at most five critical points and at most
three local maximizers. Hence, $t_-,0, t_+$ are actually the all  maximizers of $F$. Using similar argument as for Case 2c,  we can also show that $F''(t_+)=F''(t_-)<0$ and $F''(0)<0$.  Hence, the maximizers are $2$-regular. \qed

\begin{proof}[\textbf{Proof of Lemma \ref{l:gg}}]
 Observe that
\ben{
\hat{g}_\beta'(s)=\frac{1}{2s(1-s^2)}-\frac{\atanh(s)}{2s^2} -4 \beta s= 4s(q(s)-\beta), 
}
where  recall that 
\ben{ \label{eqq}
q(s)= \frac{1}{8s^2(1-s^2)} - \frac{\atanh(s)}{8s^3}.
}
Moreover,
\[
q'(s)= \frac{1}{8s^4} \left( \frac{s(5s^2-3)}{(1-s^2)^2} +3 \, \atanh(s) \right).
\]
For all $s>0$, using $\atanh(s) \geq s+s^3/3$, we get
\ben{
q'(s) \geq \frac{1}{8s^4}\left( \frac{s(5s^2-3)}{(1-s^2)^2} +3s+s^3 \right) = \frac{s^3+s}{8(1-s^2)^2} >0.
}
Hence, $q(s)$ is increasing in $(0, \infty)$. Furthermore, 
\be{
q(1) =  \infty, \qquad q(0)=1/12,
}
using $\lim_{s\rightarrow 1}(1-s) \atanh(s)=0$ and  $\atanh(s)=s+s^3/3 + \bigo(s^5)$ as $s\rightarrow 0$.
Therefore,  $q$ is an increasing function on $(0, \infty)$ taking value from $1/12$ to $\infty$.  If $\beta  \leq 1/12$ then the equation $q(s)=\beta$ has no non-zero solution. Hence, $\hat{g}_\beta$   the unique critical point $0$ which is the minimizer. If $\beta >1/12$ the equation $q(s)=\beta$ has two symmetry solutions denoted by $s_\pm$. We can check that $\hat{g}_\beta(s)>0$ for $s>s_+$ and $\hat{g}_\beta(s)<0$ for $s< s_-$. Thus $s_\pm$ are two minimizers of $\hat{g}_\beta$ and $0$ is a local maximizer, see Figure \ref{fgbt} for an illustration.   

Next, we turn to prove (ii). Since the functions $\beta \mapsto g_\beta(s)$ and $\beta \mapsto \hat{g}_\beta(s)$ are decreasing for all fixed $s$,  the functions $g$ and $\hat{g}$ are  non-increasing. By (i), if $\beta \leq 1/12$ then  $\hat{g}(\beta)=\hat{g}_\beta(0)=1/2$. Next, we consider $g(\beta)$. Since $g$ is non-increasing,
\ben{
g(1/12) \leq g(0) \leq g_0(0)=1/2.   
}
Hence, once we can show that $g(1/12) \geq 1/2$, this implies $g(\beta)=1/2$ for all $\beta \leq 1/12$. We have  
\[
g_{1/12}(s) -\frac{1}{2}= -\frac{E(s)}{s^2}- \frac{s^2}{12} - \frac{1}{2}=-\frac{l(s)}{s^2}, \quad l(s)= \frac{s^4}{12}+\frac{s^2}{2}+E(s).
\]
Observe that $l''(s)=s^4/(s^2-1)<0$ for $s \in (0,1)$. Thus $l'(s) \leq l'(0)=0$, so  $l$ is decreasing in $(0,1)$. Therefore, $l(s) \leq l(0)=0$ for all $s \in (0,1)$. Consequently,   $g_{1/12}(s) \geq 1/2$ for all $s \in (0,1)$. This together with the fact that $g$ is symmetry implies that  $g(1/12)=g_{1/12}(0)=1/2$. We conclude that 
$g(\beta) =0$ for all $\beta \leq 1/12$. 

We now prove (iii). If $h<g(\beta)$ then for all $s \in (0,1]$
\be{
F(s)=s^2(h-g_\beta(s)) <s^2(g(\beta)<g_\beta(s)) \leq 0=F(0).
}
Since $F$ is symmetry, the above inequality shows that $0$  is the unique maximizer.  Assume that $h>g(\beta)$ and  $\beta >1/12$ . We have
\ben{
g_\beta'(s)=\frac{(s-2)\log(1-s)-(s+2)\log(s+1)}{2s^3} -2 \beta s= 2s(r(s)-\beta), 
}
where 
\ben{
r(s)= \frac{(s-2)\log(1-s)-(s+2)\log(s+1)}{4s^4}.
}
Notice that $r(1)= \infty$ and  by Taylor expansion, $r(s)=1/12 + s^2/15 + \bigo(s^4)$ as $s\rightarrow 0$. Therefore, for all $\beta >1/12$ the equation $r(s)=\beta$ has at least one solution in $(0,1)$. Moreover, $r$ changes the sign from negative to positive when $s$ crosses the smallest positive solution. Therefore, $g'_\beta$ changes its sign from positive to negative at $s=0$, or $0$ is not the minimizer of $g$. Hence, there exists $s_* \in (0,1)$ such that $g(\beta)=g_\beta(s_*)$. Thus, 
\be{
F(s_*)=s_*^2(h-g_\beta(s_*))=s_*^2(h-g(\beta))>0,
}
using $h>g(\beta)$. Particularly, $0$ is  not a maximizer of $F$. 

Finally, we show that $\hat{g}(\beta) \leq g(\beta)$ for all $\beta >1/12$. Recall from \eqref{fpo} that $F'(s)=2s(h-\hat{g}_\beta(s)$. Hence,  if $h<\hat{g}(\beta)$ then $h-\hat{g}_\beta(s)<0$ for all $s$, and so  $0$ is the unique maximizer of $F$. Moreover, we have shown that if $h>g(\beta)$ and $\beta >1/12$ then $0$ is not a maximizer. Therefore, $\hat{g}(\beta) \leq g(\beta)$.
\end{proof}

\subsection*{Proof of Theorem \ref{theo:sixth}}
Thanks to our main results, we only need to study the maximizers of the associated function $A$. 
With the same arguments and notation as in the proof of Proposition \ref{prop:fours}, we now focus on analyzing the maximizers of $F: [-1,1] \rightarrow \R$ given by
\[
F(t) = \beta t^6+ ht^5+t^2/2+E(t).
\]
Note that the desired maximizers $a=1/2$ and $a=0.95$ now turn to be $t=0$ and $t=0.9$. Set 
\[
t_*=0.9.
\]
We aim to find $\beta,h$ such that $t_*$ and $0$ are maximizers of $F$. Since  $F(0)=0$, it is required  that  $F(t_*)=F'(t_*)=0$, or equivalently
\be{
\beta t_*^6+ h t_*^5+t_*^2/2+E(t_*) = 6 \beta t_*^5+5 h t_*^4+t_*+E'(t_*)=0.
}
Solving these equations, we find 
\ben{ \label{dob}
\beta =\frac{3t_*^2/2+5E(t_*)-t_*E'(t_*)}{t_*^6} \approx 0.0386,
}
\ben{ \label{doh}
h=\frac{-2t_*^2-6E(t_*)+t_*E'(t_*)}{t_*^5} \approx 0.1258.
}
It can be checked directly that $F(0)=F(t_*)=0$ and
\be{
 F'(0)=f''(0)=F'''(0)=F'(t_*)=0, \quad F^{(4)}(0), F''(t_*)<0.
}
Hence, $t=0$ and $t=t_*$ are local maximizers of $f$ with the order of regularity $4$ and $2$ respectively. Observe that for $t \neq 0$,
\be{
F(t)=t^5(h-g(t)), \qquad g(t)= -\beta t -t^{-3}/2 -E(t)t^{-5}.
}
We claim that 
\ben{ \label{gth}
g(t)>h \qquad \forall \, t\in (0,1] \setminus \{t_*\}.
}
Assuming \eqref{gth},  we  have $F(t)<0$ for all $t\in (0,1] \setminus \{0,t_*\}$. Moreover, if $t<0$ then using $g(t)=-g(-t)>h$, one has $F(t)=-(-t)^5(h+g(-t))<0$. In conclusion, 
\ben{
F(t)<0 \quad \forall \, t \in [-1,1] \setminus \{0,t_*\}.
}
Thus $0, t_*$ are actually the maximizers of $F$. Now, it remains to prove \eqref{gth}. We have 
\be{
g'(t)= t^{-6}[-\beta t^6 + 3t^2/2-tE'(t)+5E(t)].
}
Particularly,
\be{
g'(t_*)=0,
}
and 
\ben{ \label{gppt}
g''(t)=t^{-7} l(t), \qquad l(t)=-6t^2-30E(t)+10tE'(t)-t^2E''(t).
}
By direct computation,
\bea{
l'(t)=-12t -20E'(t)+8t E''(t)-t^2E'''(t)
}
and 
\bea{
l''(t)&=& -12 -12E''(t)+6tE'''(t) -t^2E^{(4)}(t)\\
&=& -12 + \frac{30}{1-t^2} - \frac{25}{(1-t^2)^2} + \frac{7}{(1-t^2)^3}\\
&=& (u-1)(7u^2 -18 u +12) >0,
}
where $u=1/(1-t^2)>1$. Therefore, $l'(t)>l'(0)=0$ for all $t>0$ and hence $l(t)>l(0)=0$ for all $t>0$. This together with \eqref{gppt} shows $g''(t)>0$ for all $t>0$ or $g$ is strictly convex  in $(0,1)$. This combining with $g'(t_*)=0$ implies that $t_*$ is the unique minimizer of $g$, or 
\[
g(t) > g(t_*)=h  \quad \forall \, t\in (0,1) \setminus \{t_*\}.
\]
Additionally, $g(1)=\infty$ and hence \eqref{gth} holds. \qed 
\section*{Acknowledgements}

This project was supported by the Singapore Ministry of Education Academic Research Fund Tier 2 grant MOE2018-T2-2-076. The work of Van Hao Can is supported by Vietnam National Foundation for Science and Technology Development (NAFOSTED) under grant number 101.03-2023.34.



\setlength{\bibsep}{0.5ex}
\def\bibfont{\small}

\end{document}